\documentclass[10pt,a4paper]{article}
\usepackage{amsfonts,amsgen,amsmath,amstext,amsbsy,amsopn,amsthm,cases,listings
}
\usepackage{multirow}
\usepackage{epstopdf}
\usepackage{graphicx}   
\usepackage{pgf,tikz}
\usepackage{mathrsfs}
\usepackage[marginal]{footmisc}
\usepackage{enumerate}
\usepackage{enumitem}
\usepackage[titletoc]{appendix}
\usepackage{booktabs}
\usepackage{url}
\usepackage{mathtools}
\usepackage{pdflscape} 
\usepackage{pgfplots}
\usepackage{authblk}
\usepackage{amssymb}
\usepackage{wasysym}
\usepackage{empheq}
\usepackage{dsfont}
\usepackage{tikz}
\usepackage{longtable}
\usepackage{float}
\usepackage{subfigure}
\pgfplotsset{compat=1.18}
\usepackage{mathrsfs}
\usepackage{wasysym} 
\usetikzlibrary{arrows}
\usepackage{aligned-overset}
\usepackage{bm}
\usepackage{bbm}
\usepackage[T1]{fontenc}
\usepackage[hidelinks]{hyperref}
\usepackage[a4paper,margin=2.5cm]{geometry}
\usepackage[thinc]{esdiff}

\allowdisplaybreaks[1]

\definecolor{uuuuuu}{rgb}{0.27,0.27,0.27}
\definecolor{sqsqsq}{rgb}{0.1255,0.1255,0.1255}

\newtheorem{definition}{Definition} [section]
\newtheorem{theorem}[definition]{Theorem}
\newtheorem{lemma}[definition]{Lemma}
\newtheorem{proposition}[definition]{Proposition}
\newtheorem{remark}[definition]{Remark}

\newtheorem{conjecture}[definition]{Conjecture}
\newtheorem{claim}[definition]{Claim}

\newtheorem{observation}[definition]{Observation}
\newtheorem{fact}[definition]{Fact}

\newcommand{\R}{\mathbb R}
\newcommand{\N}{\mathbb N}
\newcommand{\ind}{\mathbf 1}
\newcommand{\cc}{\operatorname{cc}}
\newcommand{\normsq}[1]{\|#1\|_2}
\newcommand{\norm}[1]{\left\lVert#1\right\rVert}

\usepackage{indentfirst} 
\begin{document}
\title{\bf\Large Fixed-density profiles for the semi-induced 4-vertex star}
\renewcommand\Authfont{\normalsize}
\renewcommand\Affilfont{\small}
\setlength{\affilsep}{0.35em}
\renewcommand\Authands{, }
%
\author[1]{Jinghua~Deng\thanks{Email: \texttt{Jinghua\_deng@163.com}}}
\author[1]{Jianfeng~Hou\thanks{Research supported by National Key R$\&$D Program of China (Grant No. 2023YFA1010202) and the Central Guidance on Local Science and Technology Development Fund of Fujian Province (Grant No. 2023L3003). Email: \texttt{jfhou@fzu.edu.cn}}}
\affil[1]{\small Center for Discrete Mathematics, Fuzhou University, Fuzhou, China}
\maketitle
\begin{abstract}
We study the fixed-density semi-inducibility profiles of the red-blue star
$S_{2,1}$, which has one distinguished center, two red edges and one blue
edge.  For an $n$-vertex graph $G$, let $N(S_{2,1},G)$ be the number of
injective labeled copies in which the two red edges of $S_{2,1}$ are mapped
to edges of $G$ and its blue edge is mapped to a non-edge of $G$, that is, 
\begin{align*}
    N(S_{2,1},G)=
    \sum_{v\in V(G)} d(v)(d(v)-1)(n-1-d(v)).
\end{align*}
For every fixed red edge density $\beta\in[0,1]$, we determine both
extremal $S_{2,1}$-densities.  On the upper side, we prove the missing
low-density range and, together with the theorem of Balogh, Lidick\'y,
Mubayi, Pfender and Volec for $\beta\ge 1/4$, obtain the full
four-branch profile predicted in their work.  On the lower side, we show
that the natural endpoint profile coming from the quasi-star and
quasi-clique constructions is not universal; the correct minimum is
given by a one-parameter three-class complement-split family. The proofs use a transfer argument with degree-square
tie-breaking, reducing the extremal analysis to almost-regular, threshold and
finite-staircase optimizations.

\end{abstract}





\section{Introduction}\label{introduction}
The inducibility problem, introduced  by  Pippenger and Golumbic~\cite{PippengerGolumbic},  asks for the largest possible asymptotic density of induced copies of a fixed graph. Determining the inducibility of a graph appears to be non-trivial even for simple graphs such as the four-vertex path and cycles. Exact inducibility results are known for several families and sporadic small graphs; see, for example, \cite{BrownSidorenko1994,BollobasEgawaHarrisJin1995,
HatamiHirstNorine2014,BodnarPikhurko,BodnarGaoLeonLiuPikhurkoSun2026}. On the other hand, inducibility has become a standard testing ground for modern extremal methods. Graph-limit compactness turns the problem into a variational problem over graphons~\cite{LovaszSzegedy}; flag algebras provide semidefinite certificates
for many exact bounds~\cite{Razborov}; and symmetrization and stability
arguments are used to identify and certify extremal constructions.

The semi-inducibility problem  is a red-blue analogue of inducibility. A \emph{red-blue graph} $H$ is a graph whose edges are colored red or blue; pairs of vertices which are not edges of $H$ are left unconstrained. Let $G$ be an $n$-vertex  graph, viewed as a complete red-blue graph by coloring the edges of $G$ red and the non-edges blue.  A \emph{semi-induced} copy of $H$ in $G$ is an injective map
$\phi:V(H)\to V(G)$ such that 
\begin{align*}
    \phi(x)\phi(y)\in E(G) \text{ whenever $xy$ is red},\quad\phi(x)\phi(y)\notin E(G) \text{ whenever $xy$ is blue.}
\end{align*}
Writing $|V(H)|=h$, define
\begin{align*}
    N(H,G)\coloneqq \#\{\text{semi-induced injective labeled copies of $H$ in $G$ }\}\quad \text{ and }\quad p(H,G)=\frac{N(H,G)}{n^h}.
\end{align*}
The  \emph{semi-inducibility problem} asks for the maximum limiting value of $p(H,G)$ over all $n$-vertex graphs $G$. In this language, ordinary induced-copy density is the special case in which
every pair of $H$ is colored, while ordinary subgraph density is obtained by
leaving all non-edge pairs uncolored.

The semi-inducibility problem was first studied by  Basit, Granet, Horsley, K\"undgen and Staden~\cite{BGHKS}, who established sharp or nearly sharp results for alternating walks, alternating cycles of length divisible by 4, and cycles of length 4 under every colour pattern. Very recently, Chen and Noel~\cite{CN25} determined the semi-inducibility of alternating cycles of length 6 using flag algebras, thereby resolving an open problem posed in \cite{BGHKS}.

We consider the fixed-density version of the semi-inducibility problem. For an $n$-vertex $m$-edge graph $G$, let $\beta_G\coloneqq \frac{2m}{n(n-1)}$ be the \emph{density} of $G$.  Let $\beta\in[0,1]$ and let $H$ be a fixed red-blue graph. A graph sequence $(G_n)\coloneqq(G_n)_{n=1}^{\infty}$ with $|V(G_n)|=n$, has \emph{red density $\beta$} if $\beta_{G_n}\to\beta$. The sequence $(G_n)$ is $H$-good if
$$p(H,(G_n))\coloneqq \lim_{n\to\infty} p(H,G_n)$$
exists. Define 
\begin{align*}
    I(H,\beta)=\sup p(H,(G_n)), \text{ and } i(H,\beta)=\inf p(H,(G_n))
\end{align*}
where the supremum and infimum are taken over all $H$-good sequences with red density $\beta$.

Clearly, $ I(H,\beta)$ and $ i(H,\beta)$ are the upper and the lower profiles obtained after fixing the red edge density.  A fixed-red-density question in this setting appears as Problem 9.3 in~\cite{BGHKS}. This is analogous to the feasible -region framework introduced by Liu and Mubayi~\cite{LiuMubayiFeasible}, in which one optimizes one density parameter subject to another being fixed.

Most known results on determining $I(H,\beta)$ concern small red-blue graphs
$H$, and the problem is trivial when $H$ has at most two vertices. For
three-vertex monochromatic graphs, the fixed-density values follow from
classical extremal results, namely the Kruskal--Katona
theorem~\cite{Kruskal1963,Katona1968}, the theorem of Reiher and
Wagner~\cite{ReiherWagner}, and the theorem of Ahlswede and
Katona~\cite{AhlswedeKatona}. For non-monochromatic three-vertex graphs, the
values follow from an observation of Liu, Mubayi and Reiher~\cite{LiuMubayiReiher},
together with Razborov's triangle-density theorem~\cite{RazborovTriangle} when
all three pairs are prescribed. Hence $I(H,\beta)$ is completely determined
for every red-blue graph $H$ on three vertices.


The next case is $|V(H)|=4$. For the usual semi-inducibility problem, this
case has largely been settled~\cite{BGHKS,BodnarPikhurkoSemi2025}. In the
fixed-density setting, Balogh, Lidick\'y, Mubayi, Pfender and
Volec~\cite{BLMPV} proved that
$I(AP_4,\beta)=\beta^2(1-\beta)$ for the alternating path $AP_4$, and gave
partial sharp results for the alternating four-cycle.

A further natural four-vertex example is the red-blue star $S_{2,1}$, the
star with one distinguished centre, two ordered red leaves, and one blue leaf.
For an ordinary host graph $G$, one has
\begin{align}\label{EQ:number-S2,1-degree}
N(S_{2,1},G)
=
\sum_{v\in V(G)} d(v)(d(v)-1)(n-1-d(v)),
\end{align}
where $d(v)$ is the degree of $v$. Balogh, Lidick\'y, Mubayi, Pfender and
Volec~\cite{BLMPV} considered the high-density profile of $S_{2,1}$ and
proved that

\begin{theorem}[Balogh--Lidick\'y--Mubayi--Pfender--Volec~\cite{BLMPV}]\label{thm:dense-upper-profile}
    For every $\beta\in[1/4,1]$,
    \begin{align*}
        I(S_{2,1},\beta)=
        \begin{cases}
        \beta/4, & 1/4\le \beta\le 1/2,\\[1mm]
        \beta^2(1-\beta), & 1/2\le \beta\le 1.
        \end{cases}
\end{align*}
\end{theorem}
For $\beta\in [0,1]$, let 
\begin{align*}
        s(\beta)\coloneqq
        \max\{xy^2(1-y)+y(x+y)^2(1-x-y):
        x,y\ge0,\ x+y\le1,\ 2xy+y^2=\beta\}.
\end{align*}
Let $t_*\in(0,1/3)$ be the unique root of
\begin{align}\label{eq:switching-equation}
        6t^5-384t^4+338t^3-89t^2-4t+2=0.
\end{align}

\begin{proposition}\label{lem:split-quasiclique-comparison}
If $0<\beta<t_*^2$, then $s(\beta)>\beta^{3/2}(1-\sqrt{\beta})$; if $t_*^2\le \beta\le 1/4$, then 
$s(\beta)=\beta^{3/2}(1-\sqrt{\beta})$.
\end{proposition}

The proof of Proposition~\ref{lem:split-quasiclique-comparison} is rather technical and does not rely on the structure, so we defer it to the Appendix.

For the full upper profile of $S_{2,1}$,   they \cite{BLMPV} conjectured that 

\begin{conjecture}[Balogh--Lidick\'y--Mubayi--Pfender--Volec~\cite{BLMPV}]\label{conj:upper-profile}
 There exists $y \in (0,1/4)$ such that, for every $\beta\in[0,1]$,
\begin{align*}
        I(S_{2,1},\beta)=
        \begin{cases}
        s(\beta), & 0\le \beta\le y,\\[1mm]
        \beta^{3/2}(1-\sqrt{\beta}), & y\le \beta\le 1/4,\\[1mm]
        \beta/4, & 1/4\le \beta\le 1/2,\\[1mm]
        \beta^2(1-\beta), & 1/2\le \beta\le 1. 
        \end{cases}
\end{align*}

\end{conjecture}

Our first result proves the remaining low-density part of Conjecture~\ref{conj:upper-profile}. Together with Theorem~\ref{thm:dense-upper-profile}  gives the full upper profile of $S_{2,1}$.
\begin{theorem}\label{thm:upper-profile}
For $\beta\in[0,1/4]$, let $\beta_*=t_*^2$. Then 
\begin{align*}
        I(S_{2,1},\beta)=
        \begin{cases}
        s(\beta), & 0\le \beta\le \beta_*,\\[1mm]
        \beta^{3/2}(1-\sqrt{\beta}), & \beta_*\le \beta\le 1/4.
        \end{cases}
\end{align*}
\end{theorem}

We now turn to the lower profile. Balogh, Lidick\'y, Mubayi, Pfender and
Volec~\cite{BLMPV}  also conjectured that the minimum is attained by the two endpoint
constructions, namely the quasi-star and the quasi-clique (detailed definition can be found in \cite{BLMPV}).

\begin{conjecture}[Balogh--Lidick\'y--Mubayi--Pfender--Volec~\cite{BLMPV}]\label{conj:lower-profile}
For every $\beta\in[0,1]$,
\begin{align*}
        i(S_{2,1},\beta)=
        \begin{cases}
        (1-\beta)(1-\sqrt{1-\beta})^2, & 0\le \beta\le x,\\[1mm]
        \beta^{3/2}(1-\sqrt{\beta}), & x\le \beta\le 1,
        \end{cases}
\end{align*}
where $x\in(0,1)$ is a solution to $(1-\beta)(1-\sqrt{1-\beta})^2= \beta^{3/2}(1-\sqrt{\beta})$, i.e. $16x^3-40x^2+41x-16=0$.
\end{conjecture}
Our second result shows that this endpoint formula is not always correct.
The extra construction has three vertex classes $A,B,C$.  The class
$A\cup B$ is a clique, $A$ is complete to $C$, $B$ is
anti-complete to $C$, and $C$ is independent.  Put
\begin{align*}
        |A|=hn+o(n),\qquad |A\cup B|=zn+o(n),\qquad
        |C|=(1-z)n+o(n).
\end{align*}
This construction has edge density
$
        z^2+2h(1-z),
$
and normalized $S_{2,1}$-density
\begin{align*}
        (z-h)z^2(1-z)+(1-z)h^2(1-h).
\end{align*}
Take $\beta=9/10$, $z=47/50$ and $h=41/300$. Then Conjecture~\ref{conj:lower-profile} gives $i(S_{2,1},\beta)\ge\min\{0.0467544,0.043815\}$ while 
\begin{align*}
     (z-h)z^2(1-z)+(1-z)h^2(1-h)=\frac{19600663}{450000000}=0.043557<\min\{0.0467544,0.043815\}.
\end{align*}
For fixed $\beta$, the above construction gives $h\le z$ and thus $z^2\le \beta\le z^2+2z(1-z)$. Define
\begin{align}\label{eq:Jbeta-hbeta}
      J_\beta\coloneqq[\,1-\sqrt{1-\beta},\,\sqrt{\beta}\,],
        \qquad
        h_\beta(z)\coloneqq\frac{\beta-z^2}{2(1-z)},
\end{align}
and
\begin{align}\label{eq:Cbeta}
     C_\beta(z)\coloneqq
        (z-h_\beta(z))z^2(1-z)
        +(1-z)h_\beta(z)^2(1-h_\beta(z)).
\end{align}
At $\beta=1$, we set $J_1=\{1\}$ and use the continuous convention
$C_1(1)=0$. We prove that 

\begin{theorem}\label{thm:lower}
For every $\beta\in[0,1]$, we have 
\begin{align*}
  i(S_{2,1},\beta)=\min_{z\in J_\beta} C_\beta(z).
\end{align*}
\end{theorem}

The rest of the paper is organized as follows. In Section~\ref{sec:prep}, we
collect several useful lemmas that will be used in the proofs. We prove
Theorem~\ref{thm:upper-profile} in Section~\ref{sec:upper} and
Theorem~\ref{thm:lower} in Section~\ref{sec:lower}.

\section{Preliminary}\label{sec:prep}

In this section we collect some elementary facts used throughout the paper.

\begin{fact}\cite[Theorem 9.28]{Rudin}\label{fact:ift}
Let \(U\subseteq \mathbb{R}^{2}\) be an open neighbourhood of
\((0,0)\), and let \(F:U\to\mathbb{R}\) be continuously differentiable.
Assume that
\begin{align*}
        F(0,0)=0
        \quad\text{and}\quad
         \frac{\partial F}{\partial y}(0,0)\neq 0 .
\end{align*}
Then there exist \(\varepsilon>0\), \(\delta>0\), and a unique
continuously differentiable function $\eta:(-\varepsilon,\varepsilon)\to (-\delta,\delta)$
such that
\begin{align*}
        \eta(0)=0
        \quad\text{and}\quad
        F(x,\eta(x))=0
        \quad\text{for every } |x|<\varepsilon .
\end{align*}
Moreover, after possibly decreasing \(\varepsilon\) and \(\delta\), every
solution \((x,y)\in(-\varepsilon,\varepsilon)\times(-\delta,\delta)\) of
\(F(x,y)=0\) is of the form \(y=\eta(x)\). Finally,
\begin{align*}
\eta'(x)
        =
        -\frac{\partial F}{\partial x}(x,\eta(x)){\huge /}\frac{\partial F}{\partial y}(x,\eta(x))
        \quad\text{for every } |x|<\varepsilon.
\end{align*}
\end{fact}

Put 
\begin{align*}
    c(\beta) \coloneqq \beta^{3/2}(1 -\sqrt \beta),\quad 
cc(\beta) \coloneqq (1 -\beta)(1-\sqrt{1-\beta})^2.
\end{align*}
Recall the notation  given by  \eqref{eq:Jbeta-hbeta} and \eqref{eq:Cbeta}. We have 
\begin{fact}\label{fact:qs-endpoint}
Let
$0\le\beta<1$ and put $z_\beta\coloneqq1-\sqrt{1-\beta}$.  Then 
\begin{align*}
  z_\beta\in J_\beta, \quad h_\beta(z_\beta)=z_\beta,
  \qquad
  C_\beta(z_\beta)=\cc(\beta).
\end{align*}
Consequently, for every $0\le\beta\le1$, we have
\begin{align}\label{eq:M-le-cc}
  \min_{z\in J_\beta} C_\beta(z)\le\cc(\beta).
\end{align}
\end{fact}

\begin{proof}
First note that $0\le z_\beta<1$, since $0\le\beta<1$.  Moreover
\begin{align}\label{EQ:beta=z-beta-Fact}
  \beta=1-(1-z_\beta)^2=2z_\beta-z_\beta^2.
\end{align}
By the definition of $J_\beta$,
\begin{align*}
  z_\beta=1-\sqrt{1-\beta}\in
  [\,1-\sqrt{1-\beta},\,\sqrt\beta\,]=J_\beta .
\end{align*}
By \eqref{EQ:beta=z-beta-Fact}, we obtain
\begin{align*}
  h_\beta(z_\beta)
  &=\frac{\beta-z_\beta^2}{2(1-z_\beta)}
    =\frac{2z_\beta-z_\beta^2-z_\beta^2}{2(1-z_\beta)}
    =\frac{2z_\beta(1-z_\beta)}{2(1-z_\beta)}
    =z_\beta.
\end{align*}
Substituting it into \eqref{eq:Cbeta}, the first
term vanishes and the second term gives
\begin{align*}
  C_\beta(z_\beta)
  &=(z_\beta-z_\beta)z_\beta^2(1-z_\beta)
    +(1-z_\beta)z_\beta^2(1-z_\beta)\\
  &=z_\beta^2(1-z_\beta)^2
  =(1-\sqrt{1-\beta})^2(1-\beta)
  =\cc(\beta).
\end{align*}
Taking the minimum over $J_\beta$ gives \eqref{eq:M-le-cc} for
$0\le\beta<1$.  At $\beta=1$, both sides are $0$ by the convention
$J_1=\{1\}$ and $C_1(1)=0$.
\end{proof}

\begin{fact}\label{fact:cc-square-bound}
For every $0\le\beta\le1$,
$
  \cc(\beta)\le \beta^2/4.
$
\end{fact}

\begin{proof}
Put $s=\sqrt{1-\beta}$.  Then $0\le s\le1$ and $ \beta=(1-s)(1+s).$ The change of variables gives
\begin{align*}
  \cc(\beta)
  =s^2(1-s)^2
  =\beta^2\left(\frac{s}{1+s}\right)^2.
\end{align*}
It follows from $0\le s\le1$ that $2s\le1+s$, which implies  $s/(1+s)\le1/2$. Thus, $\cc(\beta)\le \beta^2/4$. 
\end{proof}

\begin{fact}\label{fact:qc-endpoint}
For every $0\le\beta\le1$, we have 
\begin{align*}
  \min_{z\in J_\beta} C_\beta(z)\le C_\beta(\sqrt\beta)=c(\beta).
\end{align*}
\end{fact}

\begin{proof}
For $0\le\beta<1$, put $z=\sqrt\beta$.  Then $z\in J_\beta$ and, by
\eqref{eq:Jbeta-hbeta},
\begin{align*}
  h_\beta(z)=\frac{\beta-z^2}{2(1-z)}=0.
\end{align*}
Substituting this into \eqref{eq:Cbeta} gives
\begin{align*}
  C_\beta(z)=zz^2(1-z)=z^3(1-z)=\beta^{3/2}(1-\sqrt\beta)=c(\beta).
\end{align*}
At $\beta=1$, both sides are $0$ by the convention $J_1=\{1\}$ and
$C_1(1)=0$.
\end{proof}

For an integer $k$, define 
\begin{align}\label{eq:Mndef}
 \Phi_n(k)\coloneqq k(k-1)(n-1-k), \text{ and }  M_n(k)\coloneqq\Phi_n(k+1)-\Phi_n(k)=-3k^2+(2n-3)k.
\end{align}
For a graph $G$, let 
\begin{align*}
\normsq{G}=\sum_{v\in V(G)}d^2(v). 
\end{align*}
By \eqref{EQ:number-S2,1-degree}, we have 
\begin{align}\label{EQ:number-S2,1-Phi-degree}
N(S_{2,1},G)=\sum_{v\in V(G)} \Phi_n(d(v)).
\end{align}

The next lemma describes how $N(S_{2,1},G)$  and $\normsq{G}$ change when we modify a single edge in $G$.

\begin{lemma}\label{lem:transfer}
Let $u,v,w$ be three distinct vertices of $G$ with $uw\in E(G)$ and $vw\notin E(G)$.  Let
\begin{align*}
  G'\coloneqq G-uw+vw.
\end{align*}
Then
\begin{enumerate}[label=\textup{(\roman*)}]
\item $N(S_{2,1},G')-N(S_{2,1},G)=(d(u)-d(v)-1)(3d(u)+3d(v)-2n)$;
\item $\normsq{G'}-\normsq{G}=2(d(v)-d(u)+1)$.
\end{enumerate}
\end{lemma}

\begin{proof}
Note that when going from $G$ to $G'$, $\deg(v)$ increases by one, $\deg(u)$ decreases by one, and for every other vertex $w$, $\deg(w)$ remains unchanged. By \eqref{EQ:number-S2,1-Phi-degree}, we have 
\begin{align*}
  N(S_{2,1},G')-N(S_{2,1},G)=-M_n(d(u)-1)+M_n(d(v)).
\end{align*}
Using \eqref{eq:Mndef}, we get
\begin{align*}
  -M_n(d(u)-1)=3(d(u)-1)^2-(2n-3)(d(u)-1)
  =3d(u)^2-(2n+3)d(u)+2n,
\end{align*}
and
\begin{align*}
  M_n(d(v))=-3d(v)^2+(2n-3)d(v).
\end{align*}
Hence
\begin{align*}
  -M_n(d(u)-1)+M_n(d(v))=(d(u)-d(v)-1)(3d(u)+3d(v)-2n).
\end{align*}
For (ii), a direct calculations gives 
\begin{align*}
\normsq{G'}-\normsq{G}=  (d(u)-1)^2+(d(v)+1)^2-d(u)^2-d(v)^2=2(d(v)-d(u)+1), 
\end{align*}
completing the proof. 
\end{proof}

\begin{definition}
Let $G$ be a graph and let $x,y\in V(G)$.  We say that $y$ \emph{dominates} $x$ if
\begin{align*}
  N_G(x)\setminus\{y\}\subseteq N_G(y)\setminus\{x\}.
\end{align*}
Two vertices $x,y$ are \emph{twins} if
\begin{align*}
  N_G(x)\setminus\{y\}=N_G(y)\setminus\{x\}.
\end{align*}
A \emph{twin class} is a maximal set of pairwise twin vertices.
A graph is called a \emph{threshold graph} if its vertices can be ordered $v_1,v_2,\ldots,v_n$ so that $v_i$ dominates $v_j$ whenever $i<j$.  
\end{definition}

The next observation follows immediately from the definition.

\begin{observation}\label{prop:threshold-equal-degree}
Let $G$ be a threshold graph, and let $u,v$ be distinct vertices.  If
$d(u)=d(v)$, then $u, v$ are twins. If
$d(u)<d(v)$, then
\begin{align*}
  N(u)\setminus\{v\}\subset N(v)\setminus\{u\}.
\end{align*}
\end{observation}

\begin{lemma}\label{prop:threshold-classes}
    Let $G$ be a finite threshold graph.  Let
    \begin{align*}
      I(G)\coloneqq\{v\in V(G):d(v)=0\},\qquad W\coloneqq V(G)\setminus I(G).
    \end{align*}
    Partition $W$ into maximal twin classes $V_1,V_2,\ldots,V_r$, and order them so that their common degrees are  strictly decreasing. 
    Then the following statements hold:
    \begin{enumerate}[label=\textup{(\roman*)}]
        \item each $G[V_i]$ is either a clique or an independent set;
        \item for $i\ne j$, the classes $V_i,V_j$ are either complete or anti-complete to each other;
        \item for $i\ne j$, the classes $V_i$ and $V_j$ are complete to each other if and only if $i+j\le r+1$; moreover, in the degree bookkeeping below, $V_i$ contributes its own class exactly when $2i\le r+1$;
        \item if $s_i\coloneqq|V_i|$, $S_i\coloneqq s_1+\cdots+s_i$ and $S_0\coloneqq 0$, then every $v\in V_i$ has degree
        \begin{align}\label{eq:threshold-degree}
          d(v)=S_{r+1-i}-\ind_{\{2i\le r+1\}}.
        \end{align}
    \end{enumerate}
\end{lemma}

\begin{proof}
By Observation \ref{prop:threshold-equal-degree}, we have that distinct maximal twin classes
have distinct degrees, so the ordering by decreasing degree is well-defined.

First fix $i$.  If (i) failed, then there would be vertices
$u,v,w\in V_i$ such that $uv\notin E(G)$ and $uw\in E(G)$.  Since $v,w$ are twins, 
\begin{align*}
  N(w)\setminus\{v\}=N(v)\setminus\{w\}.
\end{align*}
But $u\in N(w)\setminus\{v\}$ and $u\notin N(v)\setminus\{w\}$, a
contradiction. 

Next fix $i\ne j$.  If one edge occurs between $V_i$ and $V_j$, then the twin property inside $V_i$ and inside $V_j$ forces every cross-pair to be adjacent.  Hence $V_i$ and $V_j$ are complete to each other.  If no such edge occurs, they are anti-complete.  This proves (ii).

Then we  prove (iii). Suppose that $k>1$ and $V_k$ is complete to $V_i$.  We claim that
$V_{k-1}$ is also complete to $V_i$.  Take $u\in V_k$ and $v\in V_{k-1}$.  Then $d(u)<d(v)$, and Observation \ref{prop:threshold-equal-degree} gives
\begin{align}\label{use-obser-2.7-trivial}
  N(u)\setminus\{v\}\subset N(v)\setminus\{u\}.
\end{align}

If $i\notin \{k,k-1\}$, then it follows from $V_k$ is complete to $V_i$ that $V_i\subseteq N(u)$ and then  $V_i\subseteq N(v)$ by \eqref{use-obser-2.7-trivial}. This implies that  $V_{k-1}$ is also complete to $V_i$.
If $i=k-1$, then by \eqref{use-obser-2.7-trivial}, either $|V_{k-1}|$ or $G[V_{k-1}]$ has at least one edge. In either case, $G[V_{k-1}]$ is complete by (ii). Suppose that  $i=k$.  Note that the diagonal position is reached through an actual internal edge of $G[V_k]$, which implies  $|V_k|\ge 2$. By  \eqref{use-obser-2.7-trivial}, there are edges between $V_k$ and $V_{k-1}$. By (ii), we have $V_{k-1}$ is  complete to $V_k$. Therefore, 
the indices contributing to the degree of vertices in $V_i$ are exactly
$1,\ldots,a_i$ for some $a_i\in\{1,\ldots,r\}$, and
\begin{align*}
  d(V_i)=S_{a_i}-\ind_{\{i\le a_i\}}, 
\end{align*}
where $d(V_i)$ denotes the degree of a vertex in $V_i$.

We claim that $a_1>a_2>\cdots>a_r$.  If $a_i\le a_{i+1}$ for some $i$, then
\begin{align*}
  d(V_{i+1})\ge S_{a_i}-\ind_{\{i+1\le a_i\}}
  \ge S_{a_i}-\ind_{\{i\le a_i\}}=d(V_i),
\end{align*}
because $i+1\le a_i$ implies $i\le a_i$.  This contradicts the strict degree
ordering.  Hence the $a_i$ are strictly decreasing, and since
$1\le a_i\le r$, we must have
\begin{align*}
  a_i=r+1-i\qquad (i=1,\ldots,r).
\end{align*}
Consequently, for distinct $i,j$, the classes $V_i$ and $V_j$ are complete to
each other exactly when $j\le a_i=r+1-i$, equivalently $i+j\le r+1$.  The same
forced values of $a_i$ give the diagonal bookkeeping convention stated in
(iii).

Finally, a vertex $v\in V_i$ has neighbours precisely in the contributing
classes $V_1,\ldots,V_{r+1-i}$, with $v$ itself subtracted exactly when
$2i\le r+1$.  Their total size is $S_{r+1-i}$, so
\begin{align*}
  d(v)=S_{r+1-i}-\ind_{\{2i\le r+1\}},
\end{align*}
which proves \eqref{eq:threshold-degree}.

\end{proof}

For a sequence of threshold graphs, define $V_i$ and  $S_i$ as in Lemma \ref{prop:threshold-classes}. Let  
\begin{align*}
  x_i\coloneqq\frac{S_i}{n}\qquad(0\le i\le r).
\end{align*}
Then $x_i-x_{i-1}=|V_i|/n$, and for $v\in V_i$, by Lemma \ref{prop:threshold-classes}, 
\begin{align}\label{eq:threshold-normal-degree}
  d(v)=x_{r+1-i}n+O(1).
\end{align}
Consequently,
\begin{align}\label{eq:threshold-beta-asymp}
  \beta_G=\frac{1}{n^2}\sum_{v}d(v)=\frac{1}{n^2}\sum_{i=1}^r\sum_{v\in V_i}d(v)=\sum_{i=1}^r (x_i-x_{i-1})x_{r+1-i}+O(1/n),
\end{align}
and, with $\psi(t)\coloneqq t^2(1-t)$,
\begin{align}\label{eq:threshold-F-asymp}
  \frac{N(S_{2,1},G)}{n^4}=\sum_{v\in V(G)} \Phi_n(d(v))=\sum_{i=1}^r\sum_{v\in V_i} \Phi_n(d(v))=\sum_{i=1}^r (x_i-x_{i-1})\psi(x_{r+1-i})+O(1/n).
\end{align}
When we speak of moving a point class, we mean varying the size of a whole twin
class while keeping the off-diagonal adjacency rule from
Lemma~\ref{prop:threshold-classes}, together with its diagonal degree
bookkeeping convention.

\section{Proof of Theorem \ref{thm:upper-profile}}\label{sec:upper}

In this section we prove Theorem \ref{thm:upper-profile}.  We establish the missing low-density range $0\le\beta<1/4$ by reducing chosen upper-extremal graphs to two structural branches: an almost-regular core and a threshold graph.  
Fix integers $n,m$.  Among all $n$-vertex $m$-edge graphs maximizing $N(S_{2,1},G)$, choose one maximizing $\normsq{G}$.  Call such a graph a \emph{chosen upper-extremal graph}. Our proof relies on the following three key lemmas. We establish Lemma \ref{prop:upper-dichotomy} in Subsection \ref{subsec:prove-lemma3.1}, Lemma \ref{lem:upper-core} in Subsection \ref{subsec:prove-lemma3.2}, and Lemma \ref{lem:upper-threshold-branch} in Subsection \ref{subsec:prove-lemma3.3}. For $\beta\in [0,1/4]$, recall that 
\begin{align*}
        s(\beta)\coloneqq
        \max\{xy^2(1-y)+y(x+y)^2(1-x-y):
        x,y\ge0,\ x+y\le1,\ 2xy+y^2=\beta\}, 
\end{align*}
and 
\[
 c(\beta)=\beta^{3/2}(1-\sqrt{\beta}). 
\]

\begin{lemma}\label{prop:upper-dichotomy}
Every chosen upper-extremal graph is either threshold or has the form $L\sqcup I$, where $I$ is a set of isolated vertices and all degrees inside $L$ differ by at most one.
\end{lemma}

\begin{lemma}\label{lem:upper-core}
Let \(0\le \beta\le 1/4\), and let \((G_n)\) be a graph chosen upper-extremal graph sequence with
$2e(G_n)/n^2\to \beta$.
Suppose that, for each \(n\),
$G_n=H_{s_n}\sqcup I_{n-s_n}$, 
\(I_{n-s_n}\) is a set of isolated vertices and all degrees inside \(H_{s_n}\)
differ by at most one. Then
\[
\limsup_{n\to\infty}\frac{N(S_{2,1},G_n)}{n^4}\le c(\beta).
\]
\end{lemma}

\begin{lemma}\label{lem:upper-threshold-branch}
For $0\le\beta\le1/4$, the chosen upper-extremal threshold graph $G$ with desity $\beta$   has limiting $S_{2,1}$-density at
most $s(\beta)$.
\end{lemma}

Assuming these three structural inputs, the low-density upper bound 
follows by considering the two alternatives separately.

\begin{proof}[Proof of Theorem~\ref{thm:upper-profile}]

Let 
\begin{align*}
  \rho(G)\coloneqq\frac{N(S_{2,1},G)}{|V(G)|^4},
\end{align*}
and $\beta\in [0, 1/4]$.  Let $(G_n)$ be any $S_{2,1}$-good graph sequence
with $\beta_{G_n}\to\beta$, and let $\rho_*$ be a  subsequential limit of
$\rho(G_n)$ that attains the supremum.  Replacing each $G_n$ by a chosen upper-extremal graph with the
same number of vertices and edges can only increase $\rho(G_n)$.  Hence it is
enough to bound subsequential limits of chosen upper-extremal sequences.

After passing to a subsequence, by Lemma \ref{prop:upper-dichotomy}, we may asssume that either 
$G_n=H_{s_n}\sqcup J_n,$ where $J_n$ is an independent set of isolated vertices and the degrees in
$H_{s_n}$ differ by at most one, or $G_n$ is threshold.  In the first case, Lemma~\ref{lem:upper-core} gives $\rho^*\le c(\beta)$.
Since the split feasible set of $s(\beta)$ contains the point \((x,y)=(0,\sqrt{\beta})\), we have $s(\beta)\ge c(\beta).$
Thus \(\rho_*\le s(\beta)\).  In the second case, Lemma
\ref{lem:upper-threshold-branch} gives directly
$
  \rho_*\le s(\beta).
$
By Proposition \ref{lem:split-quasiclique-comparison}, we conclude that  
\begin{align*}
  I(S_{2,1},\beta)\le  
  \begin{cases}
    s(\beta), & 0\le\beta\le\beta_*,\\
    c(\beta), & \beta_*\le\beta\le1/4.
  \end{cases}
\end{align*}

For the reverse inequality in the low-density range, we use the construction given by Balogh, Lidick\'y, Mubayi, Pfender and Volec~\cite{BLMPV}. Let $G_n$ have a vertex partition
\begin{align*}
  V(G_n)=X_n\sqcup Y_n\sqcup Z_n,
  \qquad |X_n|=xn+o(n),\quad |Y_n|=yn+o(n),
\end{align*}
where $Y_n$ is a clique, $X_n$ is independent and complete to $Y_n$, and
$Z_n$ consists of isolated vertices.  A easy conclusions (see \cite{BLMPV}) shows that 
\begin{align*}
  \beta_{G_n}\to y^2+2xy,
\end{align*}
and
\begin{align*}
  \frac{N(S_{2,1},G_n)}{n^4}
  \to xy^2(1-y)+y(x+y)^2(1-x-y).
\end{align*}
Consequently, $s(\beta)$ is the largest limiting value over these split
constructions with edge density $\beta$.
This completes the proof of Theorem \ref{thm:upper-profile}. \end{proof}

\subsection{Proof of Lemma~\ref{prop:upper-dichotomy}}\label{subsec:prove-lemma3.1}
In this subsection, we prove Lemma~\ref{prop:upper-dichotomy}. We first establish the upper transfer rule,  and then use it to force the 
dichotomy between threshold graphs and almost-regular cores with isolated 
vertices.
\begin{lemma}\label{lem:upper-transfer}
    Let $G$ be a chosen upper-extremal graph and $u,v\in V(G)$ with $d(u)\le d(v)$.  Then:
    \begin{enumerate}[label=\textup{(\alph*)}]
        \item If  $d(u)+d(v)\le 2n/3$, then $v$ dominates $u$.
        \item If $d(u)+d(v)>2n/3$, then $d(v)-d(u)\le1$.
    \end{enumerate}
\end{lemma}

\begin{proof}
    For (a), suppose that the statement fails. Choose $w\in N(u)\setminus (N(v)\cup\{v\})$, and let $G'=G-uw+vw$. Then by Lemma \ref{lem:transfer},
    \begin{align*}
      N(S_{2,1},G')-N(S_{2,1},G)=(d(u)-d(v)-1)(3d(u)+3d(v)-2n).
    \end{align*}
Here $d(u)\le d(v)$ implies $d(u)-d(v)-1<0$, and $d(u)+d(v)\le2n/3$ implies $3d(u)+3d(v)-2n\le0$.  Hence the product is nonnegative.   On the other hand,  Lemma \ref{lem:transfer} gives
    \begin{align*}
      \normsq{G'}-\normsq{G}=2(d(v)-d(u)+1)>0, 
    \end{align*}
which contradicts the choice of $G$.
    
    For (b), suppose that $d(v)\ge d(u)+2$.  Choose $w\in N(v)\setminus (N(u)\cup\{u\})$ and move $vw$ to $uw$.  In Lemma \ref{lem:transfer}, with the roles of $u$ and $v$ interchanged, the two factors are
    \begin{align*}
      d(v)-d(u)-1>0,
      \qquad
      3d(u)+3d(v)-2n>0.
    \end{align*}
    Thus the $S_{2,1}$-count increases, a contradiction.
\end{proof}

We now use these transfer rules to prove the structural lemma (Lemma~\ref{prop:upper-dichotomy}). The 
argument separates vertices of maximum degree from the remaining 
vertices and then shows that every non-core configuration is forced to 
be threshold.

\begin{proof}[Proof of Lemma~\ref{prop:upper-dichotomy}]
Let $G$ be a $n$-vertex chosen upper-extremal graph. If $\Delta(G)\le n/3$, then Lemma \ref{lem:upper-transfer}(a) applies to every ordered pair of vertices, so neighborhoods are linearly ordered by domination and $G$ is threshold.
    
In the following, assume that $\Delta(G)>n/3$.  let 
    \begin{align*}
      L\coloneqq\{v:d(v)\in\{\Delta(G),\Delta(G)-1\}\},\qquad S\coloneqq V(G)\setminus L.
    \end{align*}
For  $v\in S$, we have $|d(v)-\Delta(G)|>1$.   Lemma \ref{lem:upper-transfer} (b) gives
    \begin{align}\label{EQ:bound-d(v)-S-Lemma}
d(v)\le 2n/3-\Delta(G)<n/3.
    \end{align}

If $S$ is a  set of isolated vertice of $G$, then we are done, since $V(G)$ already admits a bipartition. Assume that $S$ contains a non-isolated vertex,
    and let
    \begin{align*}
      R\coloneqq\{v\in S:N(v)\cap S\ne\emptyset\},\qquad Z\coloneqq S\setminus R.
    \end{align*}

\begin{claim}\label{claim-G-threshold-R-nonemptyset}
If $R\ne\emptyset$, then $G$ is threshold. 
\end{claim}
\begin{proof}

First, we show $R$ is complete to $L$.  Indeed, take $v\in R$, choose $v'\in R\setminus\{v\}$ with $vv'\in E(G)$, and fix $\omega\in L$.  By \eqref{EQ:bound-d(v)-S-Lemma}, we have
    \begin{align*}
      d(v)+d(\omega)\le (2n/3-\Delta(G))+\Delta(G)=2n/3,
    \end{align*}
Applying Lemma \ref{lem:upper-transfer} (a)  to the pair $(v,\omega)$ yields that $\omega$ dominates $v$.  Thus, we have  $v'\in N(\omega)$.  Applying the same argument to $(v',\omega)$ gives $v\in N(\omega)$.  We conclude that $R$ is complete to $L$

 Then, we show $L$ is complete. Fix $v\in R$.  If $\omega_1,\omega_2\in L$ are distinct, then $\omega_2\in N(v)\setminus\{\omega_1\}$, and the same application of Lemma \ref{lem:upper-transfer}(a) to $(v,\omega_1)$ gives $\omega_2\in N(\omega_1)$.  Thus $L$ is complete.

Recall that the vertices of $R$ have degree less than $n/3$ by \eqref{EQ:bound-d(v)-S-Lemma}. Lemma \ref{lem:upper-transfer} (a) orders their neighborhoods, and so  $G[R]$ is threshold.  If $Z=\emptyset$, then placing the clique $L$ before this threshold ordering of $R$ gives a threshold ordering.  

Suppose that $Z\neq \emptyset$. If all vertices of $Z$ are isolated vertices in $G$, then place them at the end of that ordering and we are done.  Assume that $Z$ contains a non-isolated vertex.  Choose $z^*\in Z$
    with maximum degree among vertices of $Z$. We have
    $N(z^*)\subseteq L$, and let
    \begin{align*}
      X\coloneqq L\setminus N(z^*).
    \end{align*}
    For every $z\in Z\setminus\{z^*\}$, applying
    Lemma~\ref{lem:upper-transfer}(a) to $(z,z^*)$, together with the maximal
    choice of $z^*$, gives $N(z)\subseteq N(z^*)$.  Thus, $Z$ is
    anti-complete to $X$ and to $R$.  If $X=\emptyset$, then $N(z^*)=L$ and we order $V(G)$ as follows: Place 
    the clique $L$ first, then a threshold ordering of $R$, and finally the
    vertices of $Z$ ordered by their nested neighborhoods in $L$.  This is a
    threshold ordering of $G$.  Indeed, $L$ dominates $R$, since $L$ is complete
    to $L\cup R$ and vertices of $R$ have no neighbors outside $L\cup R$.  Also
    $R$ dominates $Z$: each $z\in Z$ has $N(z)\subseteq L$, while every vertex of
    $R$ is adjacent to all vertices of $L$.  The chosen orders also handle pairs inside
    $R$.
    
Suppose that $X\ne\emptyset$. For $x\in X$ and $y\in N(z^*)$,
    \begin{align*}
      d(x)=|L|-1+|R|,
      \qquad
      d(y)=|L|-1+|R|+t_y,
      \qquad t_y\coloneqq|N_Z(y)|\ge1.
    \end{align*}
Recall that all vertices of $L$ have degree either $\Delta(G)$ or $\Delta(G)-1$. We have 
    \begin{align*}
      d(x)=\Delta(G)-1,
      \qquad
      d(y)=\Delta(G),
      \qquad
      t_y=1\quad\text{for every }y\in N(z^*).
    \end{align*}
This implies that $z^*$ is the unique non-isolated vertex of $Z$.  Then the ordering 
    \begin{align*}
      N(z^*)<X<R<z^*<Z\setminus\{z^*\}
    \end{align*}
    (with the internal threshold order on $R$) is a threshold ordering.  Indeed,
    vertices in $N(z^*)$ dominate vertices in $X$, since they have the same
    neighbors in $L\setminus\{x,y\}$ and in $R$, and additionally see $z^*$.
    Every $x\in X$ dominates every $v\in R$, since $x$ is complete to
    $(L\setminus\{x\})\cup R$, while $N(v)\subseteq L\cup R$.  Every $v\in R$
    dominates $z^*$, since $N(z^*)\subseteq N(v)$, and every vertex dominates an isolated vertex.  
\end{proof}

By Claim~\ref{claim-G-threshold-R-nonemptyset}, it remains to consider $R=\emptyset$, which we treat by a strategy analogous to that of Claim~\ref{claim-G-threshold-R-nonemptyset}. Note that $S=Z$ is independent. If $Z$ is the set of isolated vertices, then we are again in the second branch. 

Suppose that $Z\neq \emptyset$. Choose $z^*\in Z$ of maximum degree among vertices of $Z$, and put $X=L\setminus N(z^*)$.  For every $z\in Z\setminus\{z^*\}$, we still have $N(z)\subseteq N(z^*)$ by  Lemma~\ref{lem:upper-transfer}(a).   Moreover every $y\in N(z^*)$ is adjacent to every other vertex of $L$: if $\omega\in L\setminus\{y\}$, then $\omega$ dominates $z^*$ by Lemma~\ref{lem:upper-transfer}(a), and $y\in N(z^*)\setminus\{\omega\}$ gives $y\in N(\omega)$. If $X=\emptyset$, then $N(z^*)=L$ and so $L$ is a clique.  Placing  $L$ first
    and then ordering $Z$ by decreasing neighborhoods into $L$ results a a threshold ordering of $G$.

Suppose that $X\ne\emptyset$. Recall that  every $y\in N(z^*)$ is adjacent to every vertex of
    $L\setminus\{y\}$.   Thus, for $x\in X$ and $y\in N(z^*)$,  $y$ sees all vertices of $L\setminus\{y\}$, while the neighbors of $x$ are
    precisely the vertices of $N(z^*)$ together with its neighbors inside $X$.
    Hence
    \begin{align*}
      d(y)=|L|-1+t_y,
      \qquad
      d(x)=|N(z^*)|+d_X(x)\le |L|-1,
    \end{align*}
    where $t_y\coloneqq|N_Z(y)|\ge1$.  By the definition of  $L$,  every $y\in N(z^*)$ has
    degree $\Delta(G)$ and every $x\in X$ has degree $\Delta(G)-1$. Thus, 
    $\Delta(G)=|L|$ and $t_y=1$ for every $y\in N(z^*)$.  Moreover
    \begin{align*}
      d_X(x)=d(x)-|N(z^*)|=|L|-1-|N(z^*)|=|X|-1,
    \end{align*}
which implies that  $X$ is a clique.  Since $t_y=1$ for every $y\in N(z^*)$, the vertices in 
    $Z\setminus\{z^*\}$  are isolated.   The ordering
    \begin{align*}
      N(z^*)<X<z^*<Z\setminus\{z^*\}
    \end{align*}
    is a threshold ordering of $G$, completing the proof. 
\end{proof}

\subsection{Proof of Lemma~\ref{lem:upper-core}}\label{subsec:prove-lemma3.2}

In this subsection, we prove Lemma~\ref{lem:upper-core}. Let \((G_n)\) be a graph sequence with
$2e(G_n)/n^2\to \beta$.
Suppose that, for each \(n\),
$G_n=H_{s_n}\sqcup I_{n-s_n}$, 
where \(I_{n-s_n}\) is the set of isolated vertices and all degrees inside \(H_{s_n}\)
differ by at most one.  It is enough to bound an arbitrary subsequential limit. Put $\sigma_n\coloneqq s_n/n$, and pass to a subsequence on which \(\sigma_n\to \sigma\).

If \(\beta=0\), then \(e(G_n)=o(n^2)\) for large $n$. Combining \eqref{eq:Mndef} and \eqref{EQ:number-S2,1-Phi-degree} gives 
\[
N(S_{2,1},G_n)
=\sum_{v\in V(G_n)}\Phi_n(d(v))
\le n\sum_{v\in V(G_n)}d(v)^2
\le 2e(G_n)n^2
=o(n^4).
\]
We are done.

Suppose that \(\beta>0\). Note that  all edges in $G_n$ lie inside \(H_{s_n}\), which implies  \(\sigma>0\). Moreover, there is an integer \(D_n\) such that for each  \(v\in V(H_{s_n})\), $d(v)\in\{D_n,D_n-1\}$.
Hence
\[
2e(G_n)=\sum_{v\in V(H_{s_n})}d(v)=D_ns_n+O(s_n).
\]
Equivalently, 
\[
\beta+o(1)
=\frac{\sigma_n}{n}D_n+o(1).
\]
This means that  for \(v\in V(H_{s_n})\),
\begin{align}\label{eq:Gn-d(v)-beta-sigma}
\frac{d(v)}{n}=\frac{\beta}{\sigma_n}+o(1).
\end{align}

On the other hand, for \(v\in V(H_{s_n})\), we have \(d(v)\le s_n\). Hence
\[
\frac{\beta}{\sigma_n}+o(1)=\frac{d(v)}{n}\le \frac{s_n}{n}=\sigma_n,
\]
and passing to the limit gives
$\beta/\sigma\le \sigma$.
Thus every possible limit \(\sigma\) satisfies
\begin{align}\label{eq:bound-beta-sigma}
\sqrt{\beta}\le \sigma\le 1. 
\end{align}

Now we count $S_{2,1}$ in $G_n$.  The isolated vertices do not contribute to \(N(S_{2,1},G_n)\). Substituting \eqref{eq:Gn-d(v)-beta-sigma} into \eqref{EQ:number-S2,1-Phi-degree} yields that 
\begin{align}\label{eq:Gn-S21-density}
\frac{N(S_{2,1},G_n)}{n^4}
=
\frac{\beta^2}{\sigma_n}
\left(1-\frac{\beta}{\sigma_n}\right)
+o(1).
\end{align}
Define
\[
f_\beta(\sigma):=\frac{\beta^2}{\sigma}\left(1-\frac{\beta}{\sigma}\right),
\qquad \sqrt{\beta}\le \sigma\le 1.
\]
By \eqref{eq:Gn-S21-density}, every subsequential limit of \(N(S_{2,1},G_n)/n^4\)  equals
\(f_\beta(\sigma)\) for some \(\sigma\in[\sqrt{\beta},1]\). It follows from \(0<\beta\le 1/4\) that  \(2\beta\le \sqrt{\beta}\le \sigma\). We have 
\[
f'_\beta(\sigma)=\frac{\beta^2(2\beta-\sigma)}{\sigma^3}\le 0,
\]
Thus, \(f_\beta\) is decreasing on
\([\sqrt{\beta},1]\). By \eqref{eq:bound-beta-sigma}, we conclude that 
\[
f_\beta(\sigma)\le f_\beta(\sqrt{\beta})
=\beta^{3/2}(1-\sqrt{\beta})
=c(\beta), 
\]
completing the proof.

\subsection{Proof of Lemma~\ref{lem:upper-threshold-branch}}\label{subsec:prove-lemma3.3}
In this subsection, we prove Lemma~\ref{lem:upper-threshold-branch}. 
    Let $G$ be a chosen upper-extremal threshold graph.  Order its vertices as $w_1,\ldots,w_n$ so that
    $d(w_1)\ge\cdots\ge d(w_n)$.  By Observation
    \ref{prop:threshold-equal-degree},
    whenever $i<j$, the vertex $w_i$ dominates $w_j$.  Let $t$ be the largest
    index with $t\ge2$ and $w_{t-1}w_t\in E(G)$; if no such index exists, set
    $t=1$.  Put
    \begin{align*}
      K=\{w_1,\ldots,w_t\},\qquad T=V(G)\setminus K.
    \end{align*}
    Then $K$ is a clique and $T$ is independent by Observation
    \ref{prop:threshold-equal-degree}.  Order
    \begin{align*}
      K=\{h_1,\ldots,h_t\},\qquad T=\{v_1,\ldots,v_s\},
    \end{align*}
    so that
    \begin{align*}
      N(v_i)\cap K=\{h_1,\ldots,h_{b_i}\},
      \qquad t\ge b_1\ge\cdots\ge b_s\ge0.
    \end{align*}
 Let $c_q\coloneqq|\{i:b_i\ge q\}|$.  Then $d(v_i)=b_i$ and $d(h_q)=t-1+c_q$. Recall that 
\begin{align*}
 \Phi_n(k)\coloneqq k(k-1)(n-1-k), \text{ and }  M_n(k)\coloneqq\Phi_n(k+1)-\Phi_n(k)=-3k^2+(2n-3)k, 
\end{align*} 
and $N(S_{2,1},G)=\sum_{v\in V(G)} \Phi_n(d(v))$ by \eqref{eq:Mndef} and \eqref{EQ:number-S2,1-Phi-degree}.

    \begin{claim}\label{fact:threshold-exact-count}
       We have that 
        \begin{align}\label{eq:threshold-exact-upper}
          N(S_{2,1},G)=t\Phi_n(t-1)+\sum_{i=1}^s\left(\Phi_n(b_i)+b_iM_n(t+i-2)\right).
        \end{align}
    \end{claim}
    
    \begin{proof}
        By  \eqref{EQ:number-S2,1-Phi-degree}, it suffices to sum $\Phi_n(d(x))$ over
        $x\in V(G)$.  Since $d(v_i)=b_i$, the vertices in $T$ contribute
        $\sum_{i=1}^s\Phi_n(b_i)$ and the vertices in $K$ contribute
        $\sum_{q=1}^t\Phi_n(t-1+c_q)$.
        
        The definition of $M_n$ gives
        \begin{align*}
          \Phi_n(t-1+c_q)
          =\Phi_n(t-1)+\sum_{j=1}^{c_q}
            \bigl(\Phi_n(t+j-1)-\Phi_n(t+j-2)\bigr)
          =\Phi_n(t-1)+\sum_{j=1}^{c_q}M_n(t+j-2).
        \end{align*}
        Therefore
        \begin{align*}
          \sum_{q=1}^t\Phi_n(d(h_q))
          =t\Phi_n(t-1)+\sum_{q=1}^t\sum_{j=1}^{c_q}M_n(t+j-2).
        \end{align*}
        For fixed $j$, the condition $j\le c_q$ is equivalent to $q\le b_j$, that is,
        to $h_qv_j\in E(G)$.  Thus, after interchanging the order of summation, the term
        $M_n(t+j-2)$ occurs exactly $b_j$ times.  Hence
        \begin{align*}
          \sum_{q=1}^t\sum_{j=1}^{c_q}M_n(t+j-2)
          =\sum_{j=1}^s b_jM_n(t+j-2).
        \end{align*}
        Combining the contributions from $K$ and $T$ proves
        \eqref{eq:threshold-exact-upper}.
    \end{proof}

    Writing $z=t/n$ and $u_i=b_i/n$, Claim \ref{fact:threshold-exact-count} gives that 
    \begin{align}\label{eq:threshold-asymp-upper}
      \frac{N(S_{2,1},G)}{n^4}=z^3(1-z)+\frac1n\sum_{i=1}^s\left[u_i^2(1-u_i)+u_i\left(z+\frac{i}{n}\right)\left(2-3\left(z+\frac{i}{n}\right)\right)\right]+O(1/n).
    \end{align}
Counting the number of edge in $G$ yields that $e(G)=\binom t2+\sum_{i=1}^s b_i$. This together with  $t=zn$ and $b_i=u_in$ gives 
    \begin{align}\label{eq:threshold-edge-upper}
      \frac{2e(G)}{n^2}
      =\frac{t(t-1)}{n^2}+\frac{2}{n^2}\sum_{i=1}^s b_i
      =z^2+\frac{2}{n}\sum_{i=1}^s u_i+O(1/n).
    \end{align}
    
    \begin{claim}\label{lem:capped-simplex}
        Fix $s\in\N$, $z\ge0$, and $0\le A\le sz$.  Let
        \begin{align*}
          P_{s,A}(z)\coloneqq\{\mathbf{u}=(u_1,\ldots,u_s)\in\R^s:
          z\ge u_1\ge\cdots\ge u_s\ge0,
          \quad \sum_{i=1}^s u_i=A\}.
        \end{align*}
        An $s$-tuple $\mathbf{u}\in P_{s,A}(z)$ is not the midpoint of two distinct
        vectors of $P_{s,A}(z)$ if and only if it has the form
        \begin{align}\label{eq:PsA-extreme-form}
          u_1=\cdots=u_a=z,
          \qquad
          u_{a+1}=\cdots=u_b=h,
          \qquad
          u_{b+1}=\cdots=u_s=0,
        \end{align}
        for some $0\le a\le b\le s$ and $0\le h\le z$, with the sum constraint
        $\sum_i u_i=A$.
    \end{claim}
    
    \begin{proof}
        Let $B_1,\ldots,B_m$ be the maximal equal-value blocks of $\mathbf{u}$, with
        block values $\lambda_1>\cdots>\lambda_m$.  If two block values
        $\lambda_\alpha,\lambda_\gamma$ lie in $(0,z)$, choose $\varepsilon>0$ small
        enough and define
        \begin{align*}
          x_i=
          \begin{cases}
            u_i+\varepsilon, & i\in B_\alpha,\\
            u_i-\dfrac{|B_\alpha|}{|B_\gamma|}\varepsilon, & i\in B_\gamma,\\
            u_i, & \text{otherwise}.
          \end{cases}
        \end{align*}
        Then $\mathbf{x}=(x_1,\ldots,x_s)\in P_{s,A}(z)$, and for $\mathbf{y}\coloneqq2\mathbf{u}-\mathbf{x}$
        we also have $\mathbf{y}\in P_{s,A}(z)$ and $\mathbf{x}\ne\mathbf{y}$.  Thus
        $\mathbf{u}$ is the midpoint of two distinct feasible vectors.  Hence a vector
        which is not such a midpoint has at most one equal-value block with value in
        $(0,z)$.  Monotonicity then forces all $z$-coordinates to form a prefix and all
        $0$-coordinates to form a suffix, giving the form in \eqref{eq:PsA-extreme-form}.
        
        Conversely, assume that $\mathbf{u}$ has the form in \eqref{eq:PsA-extreme-form} and satisfies 
        \begin{align}\label{eq:midpoint-decomp}
          2\mathbf{u}=\mathbf{x}+\mathbf{y},
          \qquad
          \mathbf{x},\mathbf{y}\in P_{s,A}(z).
        \end{align}
        The bounds force $x_i=y_i=z$ for $i\le a$ and $x_i=y_i=0$ for $i>b$.  On the
        middle block, $x_i+y_i=2h$.  Since $\mathbf{x}$ is nonincreasing and
        $\mathbf{y}$ is nonincreasing, the sequence $(x_i)_{a<i\le b}$ is both
        nonincreasing and nondecreasing.  Hence
        $x_{a+1}=\cdots=x_b$ and $y_{a+1}=\cdots=y_b$.  If $b>a$, then the sum
        constraints give
        \begin{align*}
          za+x_{a+1}(b-a)=A=za+y_{a+1}(b-a),
        \end{align*}
        so $x_{a+1}=y_{a+1}$.  Together with $x_i+y_i=2h$, this gives
        $x_i=y_i=h$ on the middle block.  If $b=a$, there is no middle block.  Thus
        $\mathbf{x}=\mathbf{y}=\mathbf{u}$, so the midpoint representation is trivial.
    \end{proof}
    
    By \eqref{eq:threshold-edge-upper}, $z^2\le\beta+o(1)$, and hence $z\le1/2+o(1)$.  For $0\le u\le z$,
    \begin{align}\label{eq:u2-bound-z}
      u^2(1-u)\le u z(1-z)+u\max\{z-1/2,0\}^2.
    \end{align}
    Indeed, if $z\le1/2$, then $u(1-u)\le z(1-z)$; if $z>1/2$, then
    $u(1-u)\le1/4=z(1-z)+(z-1/2)^2$.  Since
    $\max\{z-1/2,0\}=o(1)$, applying  \eqref{eq:u2-bound-z} with $u=u_i$ for
    $i=1,\ldots,s$ gives the error term
    \begin{align*}
      \frac1n\sum_{i=1}^s u_i\max\{z-1/2,0\}^2
      =\max\{z-1/2,0\}^2\frac1n\sum_{i=1}^s u_i=o(1),
    \end{align*}
    since $0\le u_i\le1$ and $s\le n$.
    
    Assume first that $z\ge1/3$.  The tail contribution in \eqref{eq:threshold-asymp-upper} is then at most
    \begin{align*}
      \frac1n\sum_{i=1}^s \omega_i u_i+o(1),
    \end{align*}
    where
    \begin{align*}
      \omega_i\coloneqq z(1-z)+\left(z+\frac{i}{n}\right)\left(2-3\left(z+\frac{i}{n}\right)\right).
    \end{align*}
    Since $x(2-3x)$ is strictly decreasing for $x\ge1/3$, we have
    $\omega_i>\omega_j$ whenever $i<j$.   If $i<j$ and $u_j>0$, then
    moving any sufficiently small $\varepsilon\le u_j$ from $u_j$ to
    $u_i$ changes $\sum_i\omega_i u_i$ by
    \begin{align*}
      \varepsilon(\omega_i-\omega_j)>0.
    \end{align*}
    Since each $u_i\le z$, the maximum is
    therefore attained at a vector of the form
    \begin{align*}
      z,\ldots,z,\theta,0,\ldots,0.
    \end{align*}
    There is a $t$-clique with $t=zn+o(n)$.  Suppose that the vector
    $(u_1,\ldots,u_s)$ has $\xi n+o(n)$ coordinates equal to $z$.  Then
    $\xi n+o(n)$ tail vertices are complete to the $t$-clique, and the single
    partial coordinate $\theta$ gives one tail vertex adjacent to
    $\theta n+o(n)$ vertices of the $t$-clique.  Hence, in our normalization,
    \begin{align*}
      \beta_G
      =\frac{2}{n^2}\left(\binom{t}{2}+(\xi n+o(n))t+\theta n+o(n)\right)
      =z^2+2\xi z+o(1).
    \end{align*}
    Passing to the limit gives $\beta=z^2+2\xi z$.  The partial coordinate affects
    $N(S_{2,1},G)/n^4$ by only $O(1/n)$, and the corresponding normalized
    $S_{2,1}$-count satisfies
    \begin{align*}
      \frac{N(S_{2,1},G)}{n^4}
      &=\frac{1}{n^4}\left((\xi n+o(n))\Phi_n(t)
        +t\Phi_n(t+\xi n+o(n))\right)\\
      &
      \le \xi z^2(1-z)+z(\xi+z)^2(1-\xi-z)+o(1).
    \end{align*}
    Since $2\xi z+z^2=\beta$, the pair $(x,y)=(\xi,z)$ is feasible in the
    definition of $s(\beta)$.  Hence, after taking the limit superior, this branch contributes at most $s(\beta)$.
    
    Assume next that $z<1/3$.  Put $A\coloneqq\sum_i u_i$ and
    \begin{align*}
      T_z(\mathbf{u})\coloneqq
      \frac1n\sum_{i=1}^s\left[u_i^2(1-u_i)
      +u_i\left(z+\frac{i}{n}\right)\left(2-3\left(z+\frac{i}{n}\right)\right)\right].
    \end{align*}
    For $0\le u_i\le z\le1/3$,
    \begin{align*}
      \frac{\partial^2 T_z}{\partial u_i^2}=\frac1n(2-6u_i)\ge0,
    \end{align*}
    so $T_z$ is convex on $P_{s,A}(z)$.  Choose a maximizer
    $\mathbf{u}\in P_{s,A}(z)$ with no representation
    $2\mathbf{u}=\mathbf{x}+\mathbf{y}$ by distinct
    $\mathbf{x},\mathbf{y}\in P_{s,A}(z)$; otherwise
    \begin{align*}
      T_z(\mathbf{u})\le \frac{T_z(\mathbf{x})+T_z(\mathbf{y})}2
    \end{align*}
    and one may replace $\mathbf{u}$ by $\mathbf{x}$ or $\mathbf{y}$.  Claim
    \ref{lem:capped-simplex} gives
    \begin{align*}
      \mathbf{u}=(\underbrace{z,\ldots,z}_{(p-z)n+o(n)},
      \underbrace{h,\ldots,h}_{(q-p)n+o(n)},0,\ldots,0),
    \end{align*}
    where the first block consists of tail vertices complete to the clique $K$, the
    second block consists of tail vertices with $hn+o(n)$ neighbors in $K$, and the
    remaining tail vertices are anti-complete to $K$.  Thus
    \begin{align*}
      0\le h\le z\le1/3,
      \qquad
      z\le p\le q\le1,
    \end{align*}
    and the nonzero-degree vertex classes give
    \begin{align*}
      \frac{N(S_{2,1},G)}{n^4}
      &=
      \frac{1}{n^4}\Bigl(((p-z)n+o(n))\Phi_n(zn+o(n))
      +((q-p)n+o(n))\Phi_n(hn+o(n))\\
      &\qquad\qquad
      +(hn+o(n))\Phi_n(qn+o(n))
      +((z-h)n+o(n))\Phi_n(pn+o(n))\Bigr)\\
      &=(p-z)z^2(1-z)+(q-p)h^2(1-h)+hq^2(1-q)
      +(z-h)p^2(1-p)+o(1).
    \end{align*}
    Set
    \begin{align}\label{eq:Fplus}
      F_+(p,q,h,z)\coloneqq
      (p-z)z^2(1-z)+(q-p)h^2(1-h)+hq^2(1-q)+(z-h)p^2(1-p),
    \end{align}
    and, using $t=zn+o(n)$, the edge density is
    \begin{align}\label{eq:Fplus-beta}
      \beta_G
      &=\frac{2}{n^2}\left(\binom{t}{2}
      +((p-z)n+o(n))t+((q-p)n+o(n))(hn+o(n))\right)\notag\\
      &=z^2+2z(p-z)+2h(q-p)+o(1).
    \end{align}
    Passing to the limit gives
    \begin{align*}
      \beta=2zp-z^2+2h(q-p).
    \end{align*}
    
    For fixed $(\beta,q,z)$, and for $p<q$, the edge equation determines
    \begin{align*}
      h(p)=\frac{\beta-2zp+z^2}{2(q-p)}.
    \end{align*}
    The feasible $p$-fiber is the closure in $[z,q]$ of those $p\in[z,q)$ for which $0\le h(p)\le z$.  On every relative-open part of the fiber we have $z<p<q$ and $0<h<z$.  With $\beta,q,z$ fixed,
    \begin{align}\label{eq:h-fiber-derivatives}
      h'(p)=-\frac{z-h}{q-p},
      \qquad
      h''(p)=-\frac{2(z-h)}{(q-p)^2}.
    \end{align}
    Differentiating $F_+(p,q,h(p),z)$ twice with respect to $p$, and using \eqref{eq:h-fiber-derivatives}, gives
    \begin{align}\label{eq:Fplus-second}
      \frac{d^2}{dp^2}F_+(p,q,h(p),z)=\frac{2(z-h(p))}{q-p}\left((q-p)^2+(z-h(p))(1-3h(p))\right)\ge0.
    \end{align}
    Hence no relative-interior point of this one-parameter fiber can give a strict maximum.  Therefore, for fixed $(\beta,q,z)$, a maximum on this $p$-fiber is attained when
    \begin{align*}
      h=0,
      \qquad h=z,
      \qquad p=q,
      \qquad\text{or}\qquad p=z.
    \end{align*}
    The first three faces reduce to split constructions feasible in the definition of $s(\beta)$.  Explicitly, if $h=0$, then only the full-neighborhood
    tail block of length $p-z$ remains, giving split parameters
    $(x,y)=(p-z,z)$.  If $h=z$, then the full and plateau blocks merge, giving
    $(x,y)=(q-z,z)$.  If $p=q$, the plateau has zero length and again gives
    $(x,y)=(p-z,z)$.  It remains to discuss the mixed face $p=z$.
    
    On this face write $q=z+\eta$.  Then \eqref{eq:Fplus-beta} gives
    $\beta=z^2+2h\eta$, and \eqref{eq:Fplus} becomes
    \begin{align}\label{eq:mixed-R}
      R(h,\eta,z)\coloneqq\eta h^2(1-h)+h(z+\eta)^2(1-z-\eta)+(z-h)z^2(1-z)
    \end{align}
    for the limiting $S_{2,1}$-value.  Therefore, to prove that no point on the
    mixed face exceeds the split branch, it is enough to maximize $R$ over the
    compact mixed domain
    \begin{align*}
      D_\beta=\{(h,\eta,z):0\le h\le z\le1/3,
      \ 0\le \eta\le1-z,
      \ z^2+2h\eta=\beta\}.
    \end{align*}
    If $z=0$, then $h=0$ and $\beta=0$, so this is the zero-density split degeneration.  Hence assume $z>0$.  Let 
    \begin{align*}
      x=\frac{\beta-z^2}{2z}=\frac{h\eta}{z}=\mu\eta\le\eta,
    \end{align*}
    where $\mu=h/z\in[0,1]$.
    Thus $x\ge0$ and $x+z\le\eta+z\le1$.  Consider the split graph with a clique
    of size $zn+o(n)$ and an independent class of size $xn+o(n)$ complete to this
    clique, all remaining vertices being isolated.  Its edge density is
    $z^2+2xz=\beta$, and
    \begin{align*}
      \frac{N(S_{2,1},G)}{n^4}
      &=
      \frac{1}{n^4}\left((xn+o(n))\Phi_n(zn+o(n))
      +(zn+o(n))\Phi_n((x+z)n+o(n))\right)\\
      &=xz^2(1-z)+z(x+z)^2(1-x-z)+o(1).
    \end{align*}
    Set
    \begin{align*}
      S_\beta(z)\coloneqq xz^2(1-z)+z(x+z)^2(1-x-z).
    \end{align*}
    For $h>0$, write $h=\mu z$, so $\eta=x/\mu$.  Then
    \begin{align*}
      R(\mu z,x/\mu,z)
      &=\mu z\left(1-z-\frac{x}{\mu}\right)
      \left(z+\frac{x}{\mu}\right)^2
      +x\mu z^2(1-\mu z)+(1-z)z^2(z-\mu z).
    \end{align*}
    Therefore
    \begin{align}\label{eq:mixed-diff}
      S_\beta(z)-R(h,\eta,z)
      &=zx^3(\mu^{-2}-1)+zx^2(1-\mu^{-1})
        +3z^2x^2(\mu^{-1}-1)+xz^2(1-\mu)+xz^3(\mu^2-1)\notag\\
      &=\frac{xz(1-\mu)}{\mu^2}P(\mu,z,x),
    \end{align}
    where
    \begin{align}\label{eq:Pmixed}
      P(\mu,z,x)=(\mu+1)x^2+\mu(3z-1)x+\mu^2z-\mu^2z^2-\mu^3z^2.
    \end{align}
    We first cover the boundary of $D_\beta$.  The case $z=0$ was handled above.
    If $\eta=0$, or if $h=0$ or $h=z$ (equivalently $\mu=1$ in the last case), then this is a split construction,
    possibly with zero density.  If $\eta=1-z$, equivalently $q=1$ and
    $x=\mu(1-z)$, then
    \begin{align*}
      P(\mu,z,\mu(1-z))=\mu^2\bigl((1-2z)\mu+3z(1-z)\bigr)\ge0.
    \end{align*}
    If $z=1/3$, then
    \begin{align*}
      P(\mu,1/3,x)=(\mu+1)x^2+\frac{\mu^2}{3}\left(1-\frac{1+\mu}{3}\right)\ge0.
    \end{align*}
    Since the factor $xz(1-\mu)/\mu^2\ge0$ in \eqref{eq:mixed-diff}, the case
    $\eta=1-z$ and the case $z=1/3$ give
    \begin{align*}
      S_\beta(z)-R(h,\eta,z)\ge0.
    \end{align*}
    
    At a non-boundary maximum of $R$ on $D_\beta$, we have
    \begin{align*}
      0<\mu<1,
      \qquad 0<z<1/3,
      \qquad 0<x<\mu(1-z).
    \end{align*}
    At such an interior constrained maximum, the Lagrange multiplier equations for
    \begin{align*}
      g(h,\eta,z)\coloneqq z^2+2h\eta-\beta=0
    \end{align*}
    are
    \begin{align*}
      R_h=\lambda\partial g/\partial h=2\lambda\eta,
      \qquad
      R_\eta=\lambda\partial g/\partial \eta=2\lambda h,
      \qquad
      R_z=\lambda\partial g/\partial z=2\lambda z.
    \end{align*}
    Here $R_h=\partial R/\partial h$, $R_\eta=\partial R/\partial\eta$, and
    $R_z=\partial R/\partial z$, namely
    \begin{align}\label{eq:R-partials}
      R_h&=\eta(2h-3h^2)+(z+\eta)^2(1-z-\eta)-z^2(1-z),\\
      R_\eta&=h^2(1-h)+h(z+\eta)(2-3z-3\eta),\\
      R_z&=h(z+\eta)(2-3z-3\eta)+z^2(1-z)+(z-h)(2z-3z^2).
    \end{align}
    Since $h(2\lambda\eta)=\eta(2\lambda h)$ and
    $z(2\lambda h)=h(2\lambda z)$, we have $hR_h=\eta R_\eta, zR_\eta=hR_z.$
    Substituting  \eqref{eq:R-partials} gives
    \begin{align*}
      hR_h-\eta R_\eta
      &=h\eta\left(h(1-2h)+\eta(3z+2\eta-1)\right)=0,\\
      zR_\eta-hR_z
      &=h(z-h)\left(zh-z+z^2+2\eta-6z\eta-3\eta^2\right)=0.
    \end{align*}
     Now substitute
    $h=\mu z, x=\mu\eta, \eta=x/\mu.$
    Then
    \begin{align*}
      hR_h-\eta R_\eta
      =zx\left[
        \mu z(1-2\mu z)+\frac{x}{\mu}
          \left(3z+\frac{2x}{\mu}-1\right)\right]
      =\frac{zx}{\mu^2}
        \left(2x^2+\mu(3z-1)x+\mu^3z(1-2\mu z)\right),
    \end{align*}
    and
    \begin{align*}
      zR_\eta-hR_z
      &=-\frac{z^2(1-\mu)}{\mu}
        \left(3x^2+2\mu(3z-1)x+\mu^2z(1-(\mu+1)z)\right).
    \end{align*}
    Note that $h,\eta, \mu>0$. We obtain
    \begin{align}
      E_1&\coloneqq 2x^2+\mu(3z-1)x+\mu^3z(1-2\mu z)=0,\label{eq:E1}\\
      E_2&\coloneqq 3x^2+2\mu(3z-1)x+\mu^2z(1-(\mu+1)z)=0.\label{eq:E2}
    \end{align}
    Matching coefficients gives the polynomial identity
    \begin{align}\label{eq:mixed-certificate}
      P(\mu,z,x)-\mu^3z(\mu-1)(4\mu z+z-2)
      =(2\mu-1)E_1+(1-\mu)E_2=0.
    \end{align}
    Since $0<\mu<1$ and $0<z<1/3$, both $\mu-1$ and $4\mu z+z-2$ are negative.
    Therefore $P>0$.  By \eqref{eq:mixed-diff},
    $S_\beta(z)-R(h,\eta,z)\ge0$, and thus the threshold branch never exceeds
    $s(\beta)$ for $\beta\le1/4$.

\section{Proof of Theorem \ref{thm:lower}}\label{sec:lower}

In this section we prove Theorem \ref{thm:lower}.  The proof has two parts: a degree-square estimate for the non-threshold branch, and a staircase calculation for the threshold branch.

\subsection{Auxiliary statements for the lower bound}\label{subsec:auxiliary-statements}
In this subsection, we set up the lower-extremal framework and state the three 
auxiliary lemmas used in the proof of Theorem~\ref{thm:lower}.
Recall the definitions of $J_\beta$, $h_\beta$, and $C_\beta$ from
\eqref{eq:Jbeta-hbeta} and \eqref{eq:Cbeta}.  Throughout this section put   
\begin{align*}
  \psi(t)\coloneqq t^2(1-t),
  \qquad
  M(\beta)\coloneqq\min_{z\in J_\beta}C_\beta(z).
\end{align*}
The function $M(\beta)$ is continuous on $[0,1]$.  For $0\le\beta<1$, this follows from continuity of $C_\beta(z)$ on the compact interval $J_\beta$, whose endpoints vary continuously with $\beta$.  At $\beta=1$, Fact \ref{fact:qs-endpoint} gives $0\le  M(\beta)\le cc(\beta)\to0=M(1)$.

Fix $n,m$. Among all $n$-vertex $m$-edge graphs minimizing $N(S_{2,1},G)$, choose one maximizing $\norm{G}_2$. We call it a \emph{chosen lower-extremal graph}.  The following three lemmas play a key role in our proof. We will prove Lemma \ref{prop:lower-dichotomy} in Subsection \ref{subsec:prove-lemma4.1}, Lemma \ref{lem:cauchy-branch} in Subsection \ref{subsec:prove-lemma4.2} and Lemma \ref{prop:threshold-lower} in Subsection \ref{subsec:prove-lemma4.3}. 

\begin{lemma}\label{prop:lower-dichotomy}
Let $G$ be a chosen lower-extremal graph.  Let
\begin{align*}
  S\coloneqq\{v:d(v)\le\delta(G)+1\},
  \qquad L\coloneqq V(G)\setminus S.
\end{align*}
Then either
$
  |L|\le\delta(G)+1,
$
or $G$ is a threshold graph.
\end{lemma}

The first alternative in Lemma~\ref{prop:lower-dichotomy} is handled by a degree-square 
estimate

\begin{lemma}\label{lem:cauchy-branch}
Let $(G_n)$ be a sequence of chosen lower-extremal graphs with $2e(G_n)/n^2\to\beta$. For each $G_n$, let 
\begin{align*}
    S_n\coloneqq\{v: d(v)\le\delta(G_n)+1\},\quad L_n\coloneqq V(G_n)\setminus S_n.
\end{align*} 
If $|L_n|\le\delta(G_n)+1$ for all $n$, then
\begin{align*}
 \liminf_{n\to\infty} \frac{N(S_{2,1},G_n)}{n^4}\ge M(\beta).
\end{align*}
\end{lemma}

It remains to control the threshold alternative. For this branch, 
Lemma~\ref{prop:threshold-classes} turns the graph into a finite staircase optimization.

\begin{lemma}[threshold lower theorem]\label{prop:threshold-lower}
Let $(G_n)$ be threshold graphs with $2e(G_n)/n^2\to\beta$.  Then
\begin{align*}
  \liminf_{n\to\infty}\frac{N(S_{2,1},G_n)}{n^4}
  \ge
  \min_{z\in J_\beta} C_\beta(z).
\end{align*}
\end{lemma}

\subsection{Completion of the lower proof}\label{subsec:lower-completion}
In this subsection, we complete the proof 
of Theorem~\ref{thm:lower} assuming Lemmas~\ref{prop:lower-dichotomy}-\ref{prop:threshold-lower}
\begin{proof}[Proof of Theorem \ref{thm:lower}]
Recall that $M(\beta)=\min_{z\in J_\beta}C_\beta(z)$. We first prove the construction side. If $\beta=1$, then $J_1=\{1\}$ and $C_1(1)=0$,  so complete 
graphs give $i(S_{2,1},1)\le M(1)$.

Assume that $0\le\beta<1$. Fix $z\in J_\beta$, set $h=h_\beta(z)$ and take three classes $A,B,C$ with
\begin{align*}
  |A|=hn+o(n),\qquad |B|=(z-h)n+o(n),\qquad |C|=(1-z)n+o(n).
\end{align*}
The class sizes are nonnegative since $z\in J_\beta$ gives $h\ge0$ and $h\le z$. Make $A\cup B$ a clique, make $A$ complete to $C$, make $B$ anti-complete to $C$, and make $C$ independent. The edge density tends to $z^2+2h(1-z)=\beta$.

Vertices of $A$ have degree $n-1$ and contribute 0 to the $S_{2,1}$-count. Vertices of $B$ have degree $zn+o(n)$, and vertices of C have degree $hn+o(n)$. Thus, the number of $S_{2,1}$-density of this construction is 
\begin{align*}
    \frac{1}{n^4}\sum_{v\in B\cup C}d(v)(d(v)-1)(n-1-d(v))=(z-h)z^2(1-z)+(1-z)h^2(1-h)+o(1)\to C_\beta(z).
\end{align*}
Taking $z$ to minimize $C_\beta(z)$ gives $i(S_{2,1},\beta)\le M(\beta)$.

 It remains to prove the reverse inequality. Let $(G_n)$ be any graph sequence with $2e(G_n)/n^2\to\beta$. For each $n$, let $G_n^*$ be a chosen lower-extremal graph with the same number of vertices and edges as $G_n$. Then 
 \begin{align*}
     N(S_{2,1},G_n)\ge N(S_{2,1},G_n^*),\quad \frac{2e(G_n^*)}{n^2}=\frac{2e(G_n)}{n^2}\to\beta.
 \end{align*}
 Thus it is enough to show that
 \begin{align*}
     \liminf_{n\to \infty}\frac{N(S_{2,1},G_n^*)}{n^4}\ge M(\beta).
 \end{align*}

For each $G_n^*$, define
\begin{align*}
    S_n\coloneqq \{v:d(v)\le \delta(G_n^*)+1\},\quad L_n\coloneqq V(G_n^*)\setminus S_n.
\end{align*}  
By Lemma~\ref{prop:lower-dichotomy},  for every $n$ either
\begin{align}\label{Case:L/threshold}
   |L_n|\le\delta(G_n^*)+1
  \qquad\text{or}\qquad
  G_n^*\text{ is threshold}.
\end{align}
Set
\begin{align*}
  m_\beta\coloneqq\min_{z\in J_\beta}C_\beta(z).
\end{align*}
It is enough to prove
\begin{align*}
  \liminf_{n\to\infty}\frac{N(S_{2,1},G_n^*)}{n^4}\ge m_\beta,
\end{align*}
since $N(S_{2,1},G_n)\ge N(S_{2,1},G_n^*)$.  Suppose, to the contrary, that this liminf is less than $m_\beta$.  Then there exist $m_\beta'<m_\beta$ and a subsequence $(G_{n_k}^*)_{k=1}^\infty$ such that
\begin{align*}
  \frac{N(S_{2,1},G_{n_k}^*)}{n_k^4}\to m_\beta'<m_\beta.
\end{align*}
Along this subsequence, we still have $\beta_{G_{n_k}^*}\to\beta$, and  \eqref{Case:L/threshold} remains valid for $G_{n_k}$.  If $|L_{n_k}|\le\delta(G_{n_k}^*)+1$, then by Lemma \ref{lem:cauchy-branch}, together with the continuity of $\gamma\mapsto\min_{z\in J_\gamma}C_\gamma(z)$, we have 
\begin{align*}
  \lim_{k\to\infty}\frac{N(S_{2,1},G_{n_{k}}^*)}{n_{k}^4}\ge m_\beta,
\end{align*}
a contradiction. Otherwise,   Lemma \ref{prop:threshold-lower} gives 
\begin{align*}
    \lim_{k\to\infty}\frac{N(S_{2,1},G_{n_{k}}^*)}{n_{k}^4}\ge m_\beta,
\end{align*}
again a contradiction.  It completes the proof of  Theorem \ref{thm:lower}.
\end{proof}

\subsection{Proof of Lemma~\ref{prop:lower-dichotomy}}\label{subsec:prove-lemma4.1}

In this subsection, we prove Lemma~\ref{prop:lower-dichotomy}. We first establish the lower transfer rules 
and then show that a chosen lower-extremal graph is threshold unless its 
high-degree set is small.

\begin{lemma}[lower transfer rules]\label{lem:lower-transfer}
Let $G$ be a chosen lower-extremal graph.  Then:
\begin{enumerate}[label=\textup{(\alph*)}]
\item If $d(u)+d(v)<2n/3$, then $|d(u)-d(v)|\le1$.
\item If $d(u)\le d(v)$ and $d(u)+d(v)\ge2n/3$, then
$
  N(u)\setminus\{v\}\subseteq N(v)\setminus\{u\}.
$
\end{enumerate}
\end{lemma}

\begin{proof}
For (a), assume $d(u)\ge d(v)+2$.  Then there is $w\in N(u)\setminus (N(v)\cup\{v\})$.  Move $uw$ to $vw$, and call the resulting graph $G'$.  Lemma \ref{lem:transfer} gives
\begin{align*}
  N(S_{2,1},G')-N(S_{2,1},G)=(d(u)-d(v)-1)(3d(u)+3d(v)-2n)<0.
\end{align*}
Thus $N(S_{2,1},G')<N(S_{2,1},G)$, contradicting lower extremality.

For (b), if the containment fails, choose $w\in N(u)\setminus (N(v)\cup\{v\})$ and move $uw$ to $vw$, producing $G'$.  Again Lemma \ref{lem:transfer} gives
\begin{align*}
  N(S_{2,1},G')-N(S_{2,1},G)=(d(u)-d(v)-1)(3d(u)+3d(v)-2n)\le0.
\end{align*}
Note that $d(u)-d(v)-1\le -1<0$.  If $d(u)+d(v)>2n/3$, then $N(S_{2,1},G')<N(S_{2,1},G)$, a contradiction.  If $d(u)+d(v)=2n/3$, then $N(S_{2,1},G')=N(S_{2,1},G)$, while
\begin{align*}
  \normsq{G'}-\normsq{G}=2(d(v)-d(u)+1)\ge2>0,
\end{align*}
a contradiction to the definition of a chosen lower-extremal graph.
\end{proof}

\begin{proof}[Proof of Lemma~\ref{prop:lower-dichotomy}]
Choose a vertex $x_0$ of degree $\delta(G)$.  If $u\in L$, then $d(u)\ge\delta(G)+2=d(x_0)+2$.  By Lemma \ref{lem:lower-transfer}(a), we have $d(u)+d(x_0)\ge2n/3$; otherwise it would force $|d(u)-d(x_0)|\le1$ and thus $u\in S$.  Hence every $u\in L$ satisfies
\begin{align}\label{eq:L-low}
  d(u)+\delta(G)\ge \frac{2n}{3}.
\end{align}

Assume $|L|\ge\delta(G)+2$.  

First, we claim that $S$ is independent.  Suppose instead that $xy\in E(G)$ with $x,y\in S$.  For any $u\in L$, \eqref{eq:L-low} gives
\begin{align*}
  d(x)+d(u)\ge\delta(G)+d(u)\ge\frac{2n}{3},
\end{align*}
Lemma \ref{lem:lower-transfer}(b) gives
$
  N(x)\setminus\{u\}\subseteq N(u)\setminus\{x\}.
$ 
Since $y\in N(x)\setminus\{u\}$, we get $y\in N(u)$.  Thus, 
\begin{align*}
  d(y)\ge |L|\ge\delta(G)+2,
\end{align*}
contradicting $y\in S$.

Second, order $L$ as $u_1,\ldots,u_k$ with nonincreasing degree.  We show that
$u_i$ dominates $u_j$ if $i<j$.  For $i<j$, \eqref{eq:L-low} for $u_j$ and
$d(u_i)\ge\delta(G)$ give
\begin{align*}
  d(u_i)+d(u_j)\ge\delta(G)+d(u_j)\ge\frac{2n}{3}.
\end{align*}
Thus Lemma \ref{lem:lower-transfer}(b), applied to $(u_j,u_i)$, gives
$
  N(u_j)\setminus\{u_i\}\subseteq N(u_i)\setminus\{u_j\}.
$

Third, we claim that the vertices in $S$ can be ordered as
$v_1,\ldots,v_s$ with nonincreasing degree so that $v_i$ dominates $v_j$
whenever $i<j$.  Since $S$ is independent, every neighbor of a vertex of $S$
lies in $L$.  If $v\in S$, $u_{i_0}\in N(v)$, and $j\le i_0$, then $u_{j}$ dominates
$u_{i_0}$, so $v\in N(u_j)$.  Since $d(v)\in\{\delta(G),\delta(G)+1\}$ for $v\in S$, we have $N(v)=\{u_1,\dots,u_{d(v)}\}$.  Thus
$N(v_j)\subseteq N(v_i)$ whenever $i<j$.

Finally, for every $x\in S$ and $u\in L$, we claim that $u$ dominates $x$. Note that \eqref{eq:L-low} gives
\begin{align*}
  d(x)+d(u)\ge\delta(G)+d(u)\ge\frac{2n}{3}.
\end{align*}
Lemma \ref{lem:lower-transfer}(b) yields that $u$ dominates $x$.  Hence $G$ is threshold.
\end{proof}

\subsection{Proof of Lemma~\ref{lem:cauchy-branch}}\label{subsec:prove-lemma4.2}
In this subsection, we prove Lemma~\ref{lem:cauchy-branch}. The argument bounds the normalized degree-square sum in the non-threshold 
branch and then applies a Cauchy estimate to obtain the lower profile 
$M(\beta)$.

Let
\begin{align*}
  y_v\coloneqq\frac{d(v)}{n},
  \qquad
  \sigma(G)\coloneqq\frac1n\sum_v y_v^2.
\end{align*}
For $0\le\beta\le1$, define
\begin{align}\label{eq:sigma-star}
  \sigma_*(\beta)\coloneqq 2\beta-1+(1-\beta)^{3/2}.
\end{align}

\begin{lemma}\label{lem:cauchy}
If $G_n$ is a graph sequence with $2e(G_n)/n^2\to\beta$ and
$\sigma(G_n)\le\sigma_*(\beta)+o(1)$, then
\begin{align*}
  \frac{N(S_{2,1},G_n)}{n^4}\ge \cc(\beta)-o(1).
\end{align*}
\end{lemma}

\begin{proof}
If $\beta=1$, then $\cc(1)=0$ and the claim follows from nonnegativity.  Hence
assume $\beta<1$.  For each $n$, set $G=G_n$ and write
$\beta_G=2e(G)/n^2$.  All asymptotic notation below is taken as $n\to\infty$
along this sequence; in particular, the density assumption gives
$\beta_G=\beta+o(1)$.  The equations~\eqref{eq:Mndef} and~\eqref{EQ:number-S2,1-Phi-degree} give
\begin{align*}
N(S_{2,1},G)
=\sum_v d(v)^2(n-d(v))+O(n^3).
\end{align*}
Cauchy's inequality, applied to the vectors $d(v)\sqrt{n-d(v)}$ and $\sqrt{n-d(v)}$, gives
\begin{align*}
  \left(\sum_v d(v)(n-d(v))\right)^2
  \le
  \left(\sum_v d(v)^2(n-d(v))\right)
  \left(\sum_v (n-d(v))\right).
\end{align*}
Therefore
\begin{align}\label{eq:cauchy-lower-l45}
  \frac{N(S_{2,1},G)}{n^4}=\frac{1}{n^4}\sum_v d(v)^2(n-d(v))+O(1/n)
  \ge
  \frac{\left(\sum_v d(v)(n-d(v))\right)^2}{n^4\sum_v(n-d(v))}-o(1).
\end{align}
Now
\begin{align}\label{eq:denom-l45}
  \sum_v(n-d(v))=n^2-2e(G)=(1-\beta_G+o(1))n^2=(1-\beta)n^2+o(n^2),
\end{align}
and
\begin{align*}
  \sum_v d(v)(n-d(v))
  =n\sum_v d(v)-\sum_v d(v)^2
  =2ne(G)-n^3\sigma(G)
  =(\beta_G-\sigma(G))n^3.
\end{align*}
Since, by the hypotheses, $\beta_G=\beta+o(1)$ and
$\sigma(G)=\sigma(G_n)\le\sigma_*(\beta)+o(1)$,
\begin{align}\label{eq:num-l45}
  \sum_v d(v)(n-d(v))\ge(\beta-\sigma_*(\beta))n^3-o(n^3).
\end{align}
By the definition of $\sigma_*$,
\begin{align}\label{eq:sigma-star-gap}
  \beta-\sigma_*(\beta)=1-\beta-(1-\beta)^{3/2}=(1-\beta)(1-\sqrt{1-\beta}).
\end{align}
Combining \eqref{eq:cauchy-lower-l45}, \eqref{eq:denom-l45}, \eqref{eq:num-l45}, and \eqref{eq:sigma-star-gap} gives
\begin{align*}
  \frac{N(S_{2,1},G)}{n^4}
  \ge \frac{(\beta-\sigma_*(\beta))^2}{1-\beta}-o(1)
  =(1-\beta)(1-\sqrt{1-\beta})^2-o(1)
  =\cc(\beta)-o(1).
\end{align*}
\end{proof}

\begin{proof}[Proof of Lemma~\ref{lem:cauchy-branch}]
Let $\alpha=\delta(G_n)/n$ and $y_v\coloneqq d(v)/n$.  Then $y_v\in[\alpha,1]$.  The chord inequality
\begin{align*}
  y^2\le \alpha^2+(1+\alpha)(y-\alpha)
\end{align*}
holds on this interval, since the difference between the right side and the left side is $(y-\alpha)(1-y)$.  Averaging over $y_v=d(v)/n$ gives
\begin{align}\label{eq:sigma-bound-alpha}
  \sigma(G_n)&=\frac1n\sum_v y_v^2
  \le \alpha^2+(1+\alpha)\left(\frac1n\sum_v y_v-\alpha\right)\notag\\
  &=\alpha^2+(1+\alpha)(\beta_G-\alpha)
  =\beta_G-\alpha(1-\beta_G).
\end{align}
If $|L_n|\le\delta(G_n)+1$, then vertices in $S_n$ have degree at most $\delta(G_n)+1$ and vertices in $L_n$ have degree at most $n$.  Hence
\begin{align*}
2e(G_n)&=\sum_{v\in S_n}d(v)+\sum_{v\in L_n}d(v)\\
&\le(n-|L_n|)(\delta(G_n)+1)+|L_n|n\\
&=n(\delta(G_n)+1)+|L_n|(n-\delta(G_n)-1)\\
&\le n(\delta(G_n)+1)+(\delta(G_n)+1)(n-\delta(G_n)-1).
\end{align*}
Dividing by $n^2$ gives
\begin{align}\label{eq:alpha-beta-bound}
  \beta_{G_n}\le2\alpha-\alpha^2+O(1/n).
\end{align}
Equivalently, $(1-\alpha)^2\le1-\beta_{G_n}+O(1/n)$.  Since $1-\alpha\ge0$,
\begin{align*}
  \alpha\ge1-\sqrt{1-\beta_{G_n}}-o(1).
\end{align*}
Applying this to \eqref{eq:sigma-bound-alpha} gives
\begin{align}\label{eq:sigma-upper-Gn}
  \sigma({G_n})\le \beta_{G_n}-(1-\sqrt{1-\beta_{G_n}})(1-\beta_{G_n})+o(1)
  =2\beta_{G_n}-1+(1-\beta_{G_n})^{3/2}+o(1).
\end{align}
The same Cauchy estimate used in Lemma \ref{lem:cauchy} gives
\begin{align}\label{eq:cauchy-branch-estimate}
  \frac{N(S_{2,1},G_n)}{n^4}
  \ge
  \frac{\left(\sum_v d(v)(n-d(v))\right)^2}
       {n^4\sum_v(n-d(v))}-o(1).
\end{align}
Here
\begin{align}\label{eq:Gn-denom-num}
  \sum_v(n-d(v))=(1-\beta_{G_n})n^2\quad\text{ and }\quad \sum_v d(v)(n-d(v))
  =(\beta_{G_n}-\sigma(G_n))n^3.
\end{align}
By \eqref{eq:sigma-upper-Gn}, we have
\begin{align}\label{eq:betaGn-minus-sigma-lower}
  \beta_{G_n}-\sigma(G_n)
  \ge \beta_{G_n}-\sigma_*(\beta_{G_n})-o(1)
  =(1-\beta_{G_n})(1-\sqrt{1-\beta_{G_n}})-o(1).
\end{align}
Fix $\varepsilon>0$.  If $\beta_G\le1-\varepsilon$, then substituting \eqref{eq:Gn-denom-num} and \eqref{eq:betaGn-minus-sigma-lower} into \eqref{eq:cauchy-branch-estimate} gives
\begin{align*}
  \frac{N(S_{2,1},G_n)}{n^4}\ge\frac{(\beta_{G_n}-\sigma(G_n))^2}{1-\beta_{G_n}}-o(1).
\end{align*}
Since $\sum_vd(v)(n-d(v))\ge0$, squaring  \eqref{eq:betaGn-minus-sigma-lower} loses only an $o(1)$ term.  Thus, uniformly on
$\beta_{G_n}\le1-\varepsilon$,
\begin{align*}
  \frac{N(S_{2,1},G_n)}{n^4}
  \ge (1-\beta_{G_n})(1-\sqrt{1-\beta_{G_n}})^2-o(1)
  =\cc(\beta_{G_n})-o(1).
\end{align*}
If $\beta_{G_n}>1-\varepsilon$, then nonnegativity gives
\begin{align*}
  \frac{N(S_{2,1},G_n)}{n^4}
  \ge0\ge \cc(\beta_{G_n})-\sup_{\gamma\in[1-\varepsilon,1]}\cc(\gamma),
\end{align*}
since $\beta_{G_n}\in[1-\varepsilon,1]$ implies
$
  \cc(\beta_{G_n})\le\sup_{\gamma\in[1-\varepsilon,1]}\cc(\gamma).
$
Moreover, since
\begin{align*}
  \cc(\gamma)=(1-\sqrt{1-\gamma})^2(1-\gamma),
\end{align*}
we have
\begin{align*}
  \sup_{\gamma\in[1-\varepsilon,1]}\cc(\gamma)\to0
  \qquad\text{as }\varepsilon\to0.
\end{align*}
Thus the near-complete range also gives
\begin{align*}
  \frac{N(S_{2,1},G_n)}{n^4}\ge \cc(\beta_{G_n})-o(1),
\end{align*}
after first choosing $\varepsilon$ small and then taking $n$ large.  Combining
the two ranges gives a uniform $o(1)$ error on $[0,1]$.

By Fact \ref{fact:qs-endpoint}, applied with $\beta=\beta_{G_n}$,
\begin{align*}
  \min_{z\in J_{\beta_{G_n}}} C_{\beta_{G_n}}(z)\le \cc(\beta_{G_n}).
\end{align*}
Therefore
\begin{align*}
  \frac{N(S_{2,1},G_n)}{n^4}\ge \min_{z\in J_{\beta_G}} C_{\beta_{G_n}}(z)-o(1).
\end{align*}
 Since $2e(G_n)/n^2 \to\beta $ and $M$ is continuous on $ [0,1]$,
 \begin{align*}
     M\left(\frac{2e(G_n)}{n^2}\right)\to M(\beta). 
 \end{align*}
 Therefore,
 \begin{align*}
     \liminf_{n\to\infty}\frac{N(S_{2,1},G_n)}{n^4}\ge M(\beta).
 \end{align*}
\end{proof}

\subsection{Proof of Lemma~\ref{prop:threshold-lower}}\label{subsec:prove-lemma4.3}
In this subsection, we prove Lemma~\ref{prop:threshold-lower}. 
We translate threshold graphs into finite 
staircase patterns, prove the finite dimensional inequality $F_r(x)\ge 
M(\beta_r(x))$, and then pass to graph 
sequences.

Recall that
\begin{align*}
  \psi(t)\coloneqq t^2(1-t),
  \qquad
  M(\beta)\coloneqq\min_{z\in J_\beta}C_\beta(z).
\end{align*}

The function $M$ is continuous on $[0,1]$.  For $0\le\beta<1$, this follows from continuity of $C_\beta(z)$ on the compact interval $J_\beta$, whose endpoints vary continuously with $\beta$.  At $\beta=1$, Fact \ref{fact:qs-endpoint} gives $0\le M(\beta)\le\cc(\beta)\to0=M(1)$ as $\beta\to1$.

\medskip
\noindent\textbf{Finite normalization.}
Let $G$ be a finite threshold graph on $n$ vertices.  Let $V_1,\ldots,V_r$ be its positive twin classes, ordered as in Lemma \ref{prop:threshold-classes}, and write
\begin{align*}
  s_i\coloneqq|V_i|,
  \qquad
  S_i\coloneqq s_1+\cdots+s_i,
  \qquad
  x_i^{(n)}\coloneqq\frac{S_i}{n},
  \qquad
  x_0^{(n)}\coloneqq 0.
\end{align*}
Here the $x_i^{(n)}$ are cumulative endpoints, not normalized class sizes; the normalized size of $V_i$ is the gap $x_i^{(n)}-x_{i-1}^{(n)}$.  For every finite $n$,
\begin{align*}
  0=x_0^{(n)}<x_1^{(n)}<\cdots<x_r^{(n)}\le1,
\end{align*}
where the final inequality reflects the possible presence of isolated vertices:
$x_r^{(n)}=|W|/n$, with equality exactly when $G$ has no isolated vertices.
The strict inequalities between consecutive endpoints hold for the finite
positive classes, but no such strict inequality has to survive after passing to
a closed limiting pattern.  By Lemma \ref{prop:threshold-classes}, a vertex of $V_i$ has degree
\begin{align*}
  d_i=S_{r+1-i}-\ind_{\{2i\le r+1\}}.
\end{align*}
Consequently
\begin{align}\label{eq:finite-beta}
  \frac{2e(G)}{n^2}=\frac{1}{n^2}\sum_{i=1}^{r}s_id_i=\sum_{i=1}^r\frac{s_i}{n}\frac{S_{r+1-i}}{n}+O(1/n)=\sum_{i=1}^r (x_i^{(n)}-x_{i-1}^{(n)})x_{r+1-i}^{(n)}+O(1/n),
\end{align}
and, by the identity \eqref{EQ:number-S2,1-Phi-degree},
\begin{align}\label{eq:finite-F}
  \frac{N(S_{2,1},G)}{n^4}=\sum_{i=1}^r (x_i^{(n)}-x_{i-1}^{(n)})\psi(x_{r+1-i}^{(n)})+O(1/n).
\end{align}
For the $S_{2,1}$-count, 
\begin{align*}
  \frac{d_i(d_i-1)(n-1-d_i)}{n^3}=\psi(x_{r+1-i}^{(n)})+O(1/n).
\end{align*}
Multiplying by $s_i/n$ and summing over $i$ again gives a total error $O(1/n)$, independently of the number of point classes.

For a weak closed staircase pattern $\mathbf{x}=(x_0,\ldots,x_r)$ satisfying
\begin{align*}
  0=x_0\le x_1\le\cdots\le x_r\le1
\end{align*}
put
\begin{align*}
  \Delta_i\coloneqq x_i-x_{i-1},\quad
  \beta_r(\mathbf{x})\coloneqq\sum_{i=1}^r \Delta_i x_{r+1-i},
  \quad
  F_r(\mathbf{x})\coloneqq\sum_{i=1}^r \Delta_i\psi(x_{r+1-i}).
\end{align*}
For $r=0$ we use the empty-list convention $\beta_0=F_0=0$.  The list
\begin{align*}
  (\Delta_1,x_r),(\Delta_2,x_{r-1}),\ldots,(\Delta_r,x_1)
\end{align*}
records the limiting weight of each point class together with its normalized degree level.  The quantity $\beta_r$ is the total edge-density contribution of this list, and $F_r$ is the corresponding total $S_{2,1}$-density contribution.

\begin{lemma}[zero-gap endpoint deletion]\label{lem:zero-gap}
Let $\mathbf{x}=(x_0,\ldots,x_r)$ be a weak closed staircase pattern satisfying
$
  0=x_0\le x_1\le\cdots\le x_r\le1
$
and assume that $\Delta_j=x_j-x_{j-1}=0$.  Put $q\coloneqq r+1-j$.  Remove endpoints from this finite endpoint list, not vertices from a graph, according to
\begin{align}\label{eq:delete-rule}
  D_j=
  \begin{cases}
    \{j\},&q=j\text{ or }q=j-1,\\
    \{j,q\},&\text{otherwise}.
  \end{cases}
\end{align}
Let $\widetilde{\mathbf{x}}$ be the remaining endpoint list.  Then $\widetilde{\mathbf{x}}$ is a weak closed staircase pattern of smaller length and
\begin{align*}
  \widetilde \beta(\widetilde{\mathbf{x}})=\beta_r(\mathbf{x}),
  \qquad
  \widetilde F(\widetilde{\mathbf{x}})=F_r(\mathbf{x}).
\end{align*}
\end{lemma}

For orientation, the deletion rule has the following local effect on the size-degree list.  This table is only a guide to the case check in the proof.
\begin{align*}
\begin{array}{c|c|c}
\text{case} & \text{endpoint(s) removed} & \text{only possible local change}\\ \hline
q\notin\{j,j-1\},\ j>1 & j,q &
(\Delta_q,x_j),(\Delta_{q+1},x_{j-1})\mapsto(\Delta_q+\Delta_{q+1},x_j)\\
j=1 & j,q=r & \text{last pair has degree label }x_1=0\\
q=j & j & \text{removed pair has zero weight}\\
q=j-1 & j & (\Delta_q,x_j)\mapsto(\Delta_q,x_{j-1})
\end{array}
\end{align*}
\begin{proof}
For the pattern $\mathbf{x}$, write its size-degree list as
\begin{align*}
  L(\mathbf{x})\coloneqq((\Delta_1,x_r),(\Delta_2,x_{r-1}),\ldots,(\Delta_r,x_1)).
\end{align*}
The pair $(\Delta_i,x_{r+1-i})$ contributes $\Delta_i x_{r+1-i}$ to $\beta_r(\mathbf{x})$ and $\Delta_i\psi(x_{r+1-i})$ to $F_r(\mathbf{x})$.  Since $\Delta_j=0$, we have $x_j=x_{j-1}$, and the $j$-th pair in $L(\mathbf{x})$ has zero weight.  Thus it contributes neither to $\beta_r(\mathbf{x})$ nor to $F_r(\mathbf{x})$.  The endpoint $x_j$ also appears on the anti-diagonal side of the list: since $q=r+1-j$, the $q$-th pair has degree label $x_{r+1-q}=x_j$.  Thus one must also account for the degree-label side when $x_j$ is removed.

If $r=1$, then the zero-gap assumption forces
$\Delta_1=x_1-x_0=x_1=0$, so $\mathbf{x}=(0,0)$.  The remaining endpoint list
is $(0)$, and both $\beta$ and $F$ are empty sums equal to $0$.  Hence assume
$r\ge2$.  The deletion of $x_j$ removes the zero-weight pair
\begin{align*}
  \Delta_jx_{r+1-j}=0,
  \qquad
  \Delta_j\psi(x_{r+1-j})=0.
\end{align*}
We now check the only other possible local changes.

First assume that $q\notin\{j,j-1\}$ and $j>1$.  Then $q<r$.  In the
size-degree list, the two pairs affected by deleting the endpoint $x_q$ are
$
  (\Delta_q,x_{r+1-q}), (\Delta_{q+1},x_{r-q}).
$
Since $q=r+1-j$, these are
\begin{align*}
  (\Delta_q,x_j),\qquad (\Delta_{q+1},x_{j-1}).
\end{align*}
Deleting $x_q$ does not remove any class weight from the graph-theoretic list; it only
merges the two adjacent gaps $\Delta_q$ and
$\Delta_{q+1}$
into the single gap $x_{q+1}-x_{q-1}=\Delta_q+\Delta_{q+1}$.  The new pair has
weight $\Delta_q+\Delta_{q+1}$ and degree label $x_j$.  Since the deleted
endpoint $x_j$ is a zero-gap endpoint, we have $x_j=x_{j-1}$.  Hence the
edge-density contribution and $F$-contribution are unchanged:
\begin{align*}
  \Delta_qx_j+\Delta_{q+1}x_{j-1}=(\Delta_q+\Delta_{q+1})x_j,
\quad\text{ and }\quad
  \Delta_q\psi(x_j)+\Delta_{q+1}\psi(x_{j-1})=(\Delta_q+\Delta_{q+1})\psi(x_j).
\end{align*}

If $j=1$, then $q=r$.  Here $x_1=x_0=0$, and the endpoint $x_q=x_r$ corresponds to the last pair, whose degree label is $x_1=0$.  Deleting this last pair also changes neither sum, since
\begin{align*}
  \Delta_r x_1=0,
  \qquad
  \Delta_r\psi(x_1)=0.
\end{align*}

It remains to consider the two central cases, where the zero-weight pair and the anti-diagonal role of $x_j$ overlap.  If $q=j$, then $x_j$ is the degree label of the same zero-weight pair, and deleting $x_j$ removes only
\begin{align*}
  \Delta_jx_j=0,
  \qquad
  \Delta_j\psi(x_j)=0.
\end{align*}
If $q=j-1$, then deleting $x_j$ removes the zero-weight pair
\begin{align*}
  \Delta_jx_{j-1}=0,
  \qquad
  \Delta_j\psi(x_{j-1})=0.
\end{align*}
The only nonzero pair whose degree label crosses the deleted endpoint is the $q$-th pair.  Its edge-density contribution changes from $\Delta_qx_j$ to $\Delta_qx_{j-1}$, which is the same number since $x_j=x_{j-1}$.  Likewise,
\begin{align*}
  \Delta_q\psi(x_j)=\Delta_q\psi(x_{j-1}).
\end{align*}
The remaining endpoint list is still weakly increasing, so it is a weak closed staircase pattern.  The cases above account for every pair whose weight or degree label can change; all other pairs keep the same weight and the same degree label, up to relabeling.  Therefore the two sums are unchanged.
\end{proof}

\begin{remark}
The endpoint deletion in Lemma \ref{lem:zero-gap} is used only in the finite-dimensional proof of the closed-pattern estimate \eqref{eq:finite-dimensional-estimate}.  It removes endpoints in one finite list $0=x_0\le\cdots\le x_r\le1$ and preserves $\beta_r$ and $F_r$ exactly; no finite graph is modified.  Here ``finite'' means an arbitrary natural number $r$, not a number bounded along a graph sequence.  After \eqref{eq:finite-dimensional-estimate} is proved for every finite $r$, we apply it to each graph $G_n$ with $r=r_n$, using the full normalized pattern $\mathbf{x}^{(n)}=(x_0^{(n)},\ldots,x_{r_n}^{(n)})$.
Thus an equality $x_j=x_{j-1}$ in Lemma \ref{lem:zero-gap} represents a zero
gap in the closed pattern, not a merger of two positive finite twin classes.
\end{remark}

\begin{lemma}\label{lem:first-variation}
For $0\le a<b\le1$, set
\begin{align*}
  A(a,b)\coloneqq\frac{\psi(b)-\psi(a)}{b-a}=a+b-a^2-ab-b^2.
\end{align*}
Let $\mathbf{x}=(x_1,\ldots,x_r)$ satisfy
  $0=x_0<x_1<\cdots<x_r<1.$
If $\mathbf{x}$ is a local minimum of $F_r$ among nearby tuples with the same value of $\beta_r$, then there is a constant $\Lambda$ such that
\begin{align}\label{eq:first-var-balance}
  A(x_{r-j},x_{r+1-j})+\psi'(x_j)=\Lambda,
  \qquad j=1,\ldots,r.
\end{align}
If instead $0=x_0<x_1<\cdots<x_{r-1}<x_r=1$ and the same local minimality holds with $x_r$ fixed, then the same identity holds for $j=1,\ldots,r-1$.
\end{lemma}

\begin{proof}
Let $J=\{1,\ldots,r\}$ in the case $x_r<1$, and let $J=\{1,\ldots,r-1\}$ in the boundary case $x_r=1$.  These are exactly the coordinates that may be varied.  Write
\begin{align}\label{eq:threshold-Fr-asymptotic}
  \beta_r(\mathbf{x})=\sum_{i=1}^r \Delta_i x_{r+1-i}=\sum_{i=1}^r (x_i-x_{i-1}) x_{r+1-i},\,\,
  F_r(\mathbf{x})=\sum_{i=1}^r \Delta_i\psi(x_{r+1-i})=\sum_{i=1}^r (x_i-x_{i-1})\psi(x_{r+1-i}).
\end{align}
Fix $j\in J$ and put $q\coloneqq r+1-j$.  Differentiating the summands involving $x_j$ gives, in all cases,
\begin{align}\label{eq:Fr-partial}
  \frac{\partial F_r}{\partial x_j}&=\psi(x_q)-\psi(x_{q-1})+\Delta_q\psi'(x_j)
  =\Delta_q\left(A(x_{q-1},x_q)+\psi'(x_j)\right)\notag\\&=\Delta_{r+1-j}\left(A(x_{r-j},x_{r+1-j})+\psi'(x_j)\right).
\end{align}
Replacing $\psi(t)$ by $t$ in the same calculation gives
\begin{align}\label{eq:betar-partial}
  \frac{\partial \beta_r}{\partial x_j}=2\Delta_{r+1-j}>0,
\end{align}
where the last inequality follows from  $\Delta_{r+1-j}>0$. If $|J|\le1$, there is at most one displayed quantity and hence it is constant.  Otherwise choose distinct $j,k\in J$.  Put
\begin{align*}
  H(s,\eta)\coloneqq
  \beta_r(x_1,\ldots,x_j+s,\ldots,x_k+\eta,\ldots,x_r)-\beta_r(\mathbf{x}).
\end{align*}
Then $H(0,0)=0$ and
\begin{align*}
  \frac{\partial H}{\partial\eta}(0,0)=\frac{\partial\beta_r}{\partial x_k}
  =2\Delta_{r+1-k}>0.
\end{align*}
By Fact \ref{fact:ift}, there is a differentiable function $\eta(s)$ with
$\eta(0)=0$ and $H(s,\eta(s))=0$ for all sufficiently small $|s|$.  Define the
nearby vector $\widetilde{\mathbf{x}}(s)$ by
\begin{align*}
  \widetilde{x}_i(s)=x_i\quad(i\notin\{j,k\}),\qquad
  \widetilde{x}_j(s)=x_j+s,\qquad
  \widetilde{x}_k(s)=x_k+\eta(s),
\end{align*}
so that
$
  \beta_r(\widetilde{\mathbf{x}}(s))=\beta_r(\mathbf{x}).
$
After shrinking $|s|$ if necessary, the vector $\widetilde{\mathbf{x}}(s)$ stays in the
same strict region.

Along the path $s\mapsto\widetilde{\mathbf{x}}(s)$ the edge density is fixed exactly.
Since $\mathbf{x}$ is a local minimum under this fixed-density constraint, the
one-variable function $s\mapsto F_r(\widetilde{\mathbf{x}}(s))$ has derivative $0$ at
$s=0$.  Thus
\begin{align*}
  0=\frac{\partial F_r}{\partial x_j}
  +\frac{\partial F_r}{\partial x_k}\eta'(0).
\end{align*}        
On the other hand, differentiating the identity
$\beta_r(\widetilde{\mathbf{x}}(s))=\beta_r(\mathbf{x})$ at $s=0$ gives
\begin{align*}
  0=\frac{\partial \beta_r}{\partial x_j}
  +\frac{\partial \beta_r}{\partial x_k}\eta'(0).
\end{align*}
Eliminating $\eta'(0)$ gives
\begin{align*}
  \frac{\partial F_r/\partial x_j}{\partial \beta_r/\partial x_j}
  =
  \frac{\partial F_r/\partial x_k}{\partial \beta_r/\partial x_k}.
\end{align*}
Using \eqref{eq:Fr-partial} and \eqref{eq:betar-partial}, this becomes
\begin{align*}
  A(x_{r-j},x_{r+1-j})+\psi'(x_j)
  =A(x_{r-k},x_{r+1-k})+\psi'(x_k).
\end{align*}
Since $j,k\in J$ were arbitrary, all these quantities are equal to a common constant $\Lambda$.
\end{proof}

\begin{lemma}\label{lem:one-two-gaps}
If a strict threshold pattern $\mathbf{x}$ with $x_r<1$ has at most two positive gaps, then $F_r(\mathbf{x})\ge M(\beta_r(\mathbf{x}))$.
\end{lemma}

\begin{proof}
If there is one positive gap, the graph is a clique plus isolated vertices.  Then $\beta=x_1^2$ and
\begin{align*}
  F=x_1\psi(x_1)=x_1^3(1-x_1)=c(\beta)\ge M(\beta),
\end{align*}
where the last inequality follows from Fact \ref{fact:qc-endpoint}.
Now assume that there are two positive gaps.  Write $0<x_1<x_2<1$.  Then
\begin{align*}
    \beta(x_1,x_2)=x_1x_2+(x_2-x_1)x_1=2x_1x_2-x_1^2\quad\text{ and }\quad
  F(x_1,x_2)=(x_2-x_1)\psi(x_1)+x_1\psi(x_2).
  \end{align*}
For fixed $\beta$, regard $x_1$ as the free first level and set
\begin{align*}
  x_2(x_1)\coloneqq\frac{\beta+x_1^2}{2x_1}=\frac{x_1}{2}+\frac{\beta}{2x_1}.
\end{align*}
The corresponding one-variable two-gap value is
\begin{align}\label{eq:Fbeta-two}
  F_\beta(x_1)\coloneqq F(x_1,x_2(x_1))=\left(\frac{\beta-x_1^2}{2x_1}\right)\psi(x_1)+x_1\psi\left(\frac{\beta+x_1^2}{2x_1}\right).
\end{align}
The conditions $x_1<x_2(x_1)<1$ are equivalent to
\begin{align}\label{eq:two-gap-range}
  1-\sqrt{1-\beta}<x_1<\sqrt\beta.
\end{align}
Indeed, $x_1<x_2(x_1)$ if and only if $x_1^2<\beta$. Moreover,
$x_2(x_1)<1$ if and only if $x_1>1-\sqrt{1-\beta}.$ If $x_2=x_1$, then $x_1=\sqrt{\beta}$. If $x_2=1$, then $x_1=1-\sqrt{1-\beta}$.
Thus the two endpoints are not strict two-gap points; they are boundary
degenerations of the closed interval.  Their boundary values are
\begin{align*}
  F(1-\sqrt{1-\beta},1)=(1-\beta)(1-\sqrt{1-\beta})^2=\cc(\beta),
  \qquad
  F(\sqrt\beta,\sqrt\beta)=\beta^{3/2}(1-\sqrt{\beta})=c(\beta).
\end{align*}
By Facts \ref{fact:qs-endpoint} and \ref{fact:qc-endpoint},
\begin{align}\label{eq:M-below-endpoints}
  M(\beta)\le\min\{\cc(\beta),c(\beta)\}.
\end{align}
It remains to prove that there is no $x_1$ in \eqref{eq:two-gap-range} such
that
\begin{align*}
  F_\beta(x_1)<\min\{\cc(\beta),c(\beta)\}.
\end{align*}
Suppose otherwise.  Since the two boundary degenerations have values
$\cc(\beta)$ and $c(\beta)$, the function $F_\beta$ has an interior minimizer
$u$ with
\begin{align*}
  F_\beta(u)<\min\{\cc(\beta),c(\beta)\}.
\end{align*}
Put $v\coloneqq x_2(u)$.  Then
\begin{align*}
  F_\beta'(u)=0,
  \qquad
  F_\beta''(u)\ge0,
  \qquad
  0<u<v<1.
\end{align*}
Differentiation gives
\begin{align*}
  F_\beta'(x_1)=\frac{(x_1^2-\beta)R_\beta(x_1)}{4x_1^3},
\end{align*}
where
\begin{align*}
  R_\beta(x_1)=-\beta^2-\beta x_1^2+\beta x_1+6x_1^4-3x_1^3.
\end{align*}
At $x_1=u$, we have $u^2-\beta=2u(u-v)<0$.  Hence $R_\beta(u)=0$.  Substituting $\beta=2uv-u^2$ into $R_\beta(u)=0$ gives
\begin{align*}
  R_{2uv-u^2}(u)=2u^2(-2v^2+uv+v+3u^2-2u)=0.
\end{align*}
Since $u>0$, this gives
\begin{align}\label{eq:two-gap-eq1}
  -2v^2+uv+v+3u^2-2u=0.
\end{align}
Write $H_\beta(x_1)=(x_1^2-\beta)/(4x_1^3)$.  Since $F_\beta'(x_1)=H_\beta(x_1)R_\beta(x_1)$ and $R_\beta(u)=0$, we have
\begin{align*}
  F_\beta''(x_1)=H_\beta(x_1)R_\beta'(x_1)+H_\beta'(x_1)R_\beta(x_1)\quad\text{and}\quad F_\beta''(u)=H_\beta(u)R_\beta'(u).
\end{align*}
Since $u^2-\beta=2u(u-v)<0$, we have $H_\beta(u)<0$. Note that $F_\beta''(u)\ge0$.  Therefore $R_\beta'(u)\le0$.  Since $\beta=2uv-u^2$, we have
\begin{align}\label{eq:two-gap-eq2}
  R_\beta'(u)=-2\beta u+\beta+24u^3-9u^2=2u(v(1-2u)+13u^2-5u)\le0,
\end{align}
We claim that $u+v<1$.  First $u<1/2$.  Indeed, if $u\ge1/2$, then $1-2u\le0$, and since $v<1$, we have
\begin{align*}
  (1-2u)+13u^2-5u\ge(1-2u)+13u^2-5u=13\left(u-\frac{7}{26}\right)^2+\frac{3}{52}\ge\frac34>0,
\end{align*}
a contradiction to \eqref{eq:two-gap-eq2}.  Hence $u<1/2$.  Let
\begin{align*}
  Q(z)\coloneqq-2z^2+uz+z+3u^2-2u.
\end{align*}
Then \eqref{eq:two-gap-eq1} says $Q(v)=0$, while
$
  Q(1-u)=2u-1<0,
$
and for every $z\ge1-u$,
\begin{align*}
  Q'(z)=-4z+u+1\le -4(1-u)+u+1=5u-3<0.
\end{align*}
So $Q(z)<0$ for every $z\ge1-u$.  Since $Q(v)=0$, we must have $v<1-u$, that is, $u+v<1$.

Therefore $1-v>u$, and
\begin{align*}
  F_\beta(u)=(v-u)u^2(1-u)+uv^2(1-v)\ge uv^2(1-v)>u^2v^2.
\end{align*}
Also $\beta=2uv-u^2\le2uv$, so
\begin{align*}
  u^2v^2\ge \frac{\beta^2}{4}.
\end{align*}
By Fact \ref{fact:cc-square-bound}, $\cc(\beta)\le\beta^2/4$.  Thus
\begin{align*}
  F_\beta(u)>\cc(\beta)\ge\min\{\cc(\beta),c(\beta)\},
\end{align*}
contradicting the choice of $u$.  Consequently every $x_1$ in \eqref{eq:two-gap-range} satisfies
\begin{align*}
  F_\beta(x_1)\ge\min\{\cc(\beta),c(\beta)\}\ge M(\beta).
\end{align*}
This proves the two-gap case.
\end{proof}

\begin{lemma}\label{lem:three-gaps}
If a strict threshold pattern $\mathbf{x}$ has three positive gaps and $x_3<1$, then $F_3(\mathbf{x})\ge M(\beta_3(\mathbf{x}))$.
\end{lemma}

\begin{proof}
Let $0<x_1<x_2<x_3<1$.  Then
\begin{align}\label{eq:three-beta}
  \beta=x_1x_3+(x_2-x_1)x_2+(x_3-x_2)x_1=x_2^2+2x_1(x_3-x_2).
\end{align}
Since $\beta-x_2^2=2x_1(x_3-x_2)>0$,
\begin{align}\label{eq:three-domain}
  x_3=x_2+\frac{\beta-x_2^2}{2x_1},
  \qquad
  \frac{\beta-x_2^2}{2(1-x_2)}<\frac{\beta-x_2^2}{2(x_3-x_2)}=x_1<x_2.
\end{align}
The normalized $S_{2,1}$-density, namely the value of $F_3(\bf x)$, is
\begin{align}\label{eq:Phi-three}
  \Phi_{x_2,\beta}(x_1)\coloneqq\frac{\beta-x_2^2}{2x_1}\psi(x_1)
  +x_1\psi\left(x_2+\frac{\beta-x_2^2}{2x_1}\right)
  +(x_2-x_1)\psi(x_2).
\end{align}
We compare $\Phi_{x_2,\beta}$ with the two endpoint degenerations of the interval in \eqref{eq:three-domain}.  By~\eqref{eq:Jbeta-hbeta}, if
\begin{align*}
  x_1=\frac{\beta-x_2^2}{2(1-x_2)}=h_\beta(x_2),
\end{align*}
then $x_3=1$. Using the definition of $C_\beta(x)$ in~\eqref{eq:Cbeta} and $\psi(1)=0$, we have
\begin{align*}
  \Phi_{x_2,\beta}(h_\beta(x_2))=(1-x_2)\psi(h_\beta(x_2))+(x_2-h_\beta(x_2))\psi(x_2)=C_\beta(x_2).
\end{align*}
At the right endpoint $x_1=x_2$,
\begin{align*}
  \Phi_{x_2,\beta}(x_2)=\frac{\beta-x_2^2}{2x_2}\psi(x_2)+x_2\psi\left(\frac{\beta+x_2^2}{2x_2}\right)=F_\beta(x_2),
\end{align*}
where $F_\beta$ is the two-gap value from \eqref{eq:Fbeta-two}.
Set
\begin{align*}
  T\coloneqq(\beta-x_2^2)(1-2x_2)+6x_2^2(1-x_2)^2.
\end{align*}
Over the common denominator $8x_2^2(1-x_2)^2$, expansion gives
\begin{align}\label{eq:three-endpoint-diff}
  F_\beta(x_2)-C_\beta(x_2)=\frac{(\beta-x_2^2)(2x_2(1-x_2)-(\beta-x_2^2))T}{8x_2^2(1-x_2)^2}.
\end{align}
Indeed, $\beta-x_2^2>0$ by \eqref{eq:three-beta}, and \eqref{eq:three-domain} gives
$
  \beta-x_2^2<2x_1(1-x_2)<2x_2(1-x_2).
$
Hence the sign of $F_\beta(x_2)-C_\beta(x_2)$ is the sign of $T$.

First assume $T\ge0$.  Then $C_\beta(x_2)\le F_\beta(x_2)$.  Expanding \eqref{eq:Phi-three} and subtracting $C_\beta(x_2)$ gives
\begin{align}\label{eq:three-C-diff}
  \Phi_{x_2,\beta}(x_1)-C_\beta(x_2)
  =\frac{(\beta-x_2^2)(2x_1(1-x_2)-(\beta-x_2^2))N}{8x_1^2(1-x_2)^2},
\end{align}
where
\begin{align*}
N={}&(\beta-x_2^2)((1-x_2)^2-x_1^2)
+2x_1^2(1-x_2)(1-x_1)+4x_1x_2(1-x_2)^2.
\end{align*}
The factors outside $N$ in \eqref{eq:three-C-diff} are nonnegative by \eqref{eq:three-beta} and \eqref{eq:three-domain}.  If $x_1\le1-x_2$, then $N\ge0$.  It remains in this case to consider $x_1>1-x_2$.  Then $x_2>1/2$, and $N$ decreases as the quantity $\beta-x_2^2$ increases.  From $T\ge0$,
\begin{align*}
  \beta-x_2^2\le \frac{6x_2^2(1-x_2)^2}{2x_2-1}.
\end{align*}
Therefore
\begin{align}\label{eq:N-lower}
  N\ge
  \left.N\right|_{\beta-x_2^2=\frac{6x_2^2(1-x_2)^2}{2x_2-1}}
  =\frac{2(x_2-x_1)(x_2-1)M_0}{2x_2-1},
\end{align}
where
\begin{align*}
  M_0=-(2x_2-1)x_1^2-x_1(1-x_2)\left(3\left(x_2-\frac{1}{3}\right)^2+\frac23\right)-3x_2(1-x_2)^3<0.
\end{align*}
The lower bound in \eqref{eq:N-lower} is nonnegative, since $x_2-x_1>0$, $x_2-1<0$, $2x_2-1>0$ and $M_0<0$.  Hence $N\ge0$, and $\Phi_{x_2,\beta}(x_1)\ge C_\beta(x_2)$.

Now assume $T\le0$.  Then $F_\beta(x_2)\le C_\beta(x_2)$.  Expanding \eqref{eq:Phi-three} and subtracting $F_\beta(x_2)$ gives
\begin{align}\label{eq:three-F-diff}
  \Phi_{x_2,\beta}(x_1)-F_\beta(x_2)
  =\frac{(\beta-x_2^2)(x_2-x_1)W}{8x_1^2x_2^2},
\end{align}
where
\begin{align*}
W={}&-(x_1+x_2)(\beta-x_2^2)^2+2(\beta-x_2^2)x_1x_2(1-3x_2)+4x_1^2x_2^2(x_1+x_2-1).
\end{align*}
Since $T\le 0$, $\beta-x_2^2\ge0$ and $6x_2^2(1-x_2)^2\ge0$, we have $x_2>1/2$.  The assumption $T\le0$ gives
\begin{align*}
  \beta-x_2^2\ge \frac{6x_2^2(1-x_2)^2}{2x_2-1}.
\end{align*}
Together with \eqref{eq:three-domain}, this places $\beta-x_2^2$ in the interval
\begin{align}\label{eq:b-x2-bound}
  \frac{6x_2^2(1-x_2)^2}{2x_2-1}\le \beta-x_2^2\le 2x_1(1-x_2).
\end{align}
The expression $W$ is a concave quadratic in $\beta-x_2^2$, so its minimum on this interval is attained at an endpoint.  Put
\begin{align*}
  P\coloneqq x_1(2x_2-1)-3x_2^2(1-x_2).
\end{align*}
Substituting the two endpoint values of $\beta-x_2^2$ into $W$ gives
\begin{align*}
  \left.W\right|_{\beta-x_2^2=\frac{6x_2^2(1-x_2)^2}{2x_2-1}}
  =-\frac{4x_2^2PM_0}{(2x_2-1)^2}
\quad\text{ and }\quad
  \left.W\right|_{\beta-x_2^2=2x_1(1-x_2)}=4x_1^2P.
\end{align*}

Since $1>x_2>1/2$, we have $2x_2-1>0$ and $2(1-x_2)>0$.  Comparing the two
endpoints in \eqref{eq:b-x2-bound} and multiplying by
$(2x_2-1)/(2(1-x_2))$ gives
$
  3x_2^2(1-x_2)\le x_1(2x_2-1).
$
Thus $P\ge0$.
Since $M_0<0$, both endpoint values of $W$ are nonnegative.  Therefore $W\ge0$ throughout the interval, and $\Phi_{x_2,\beta}(x_1)\ge F_\beta(x_2)$.

Thus in all cases
\begin{align*}
  \Phi_{x_2,\beta}(x_1)\ge\min\{C_\beta(x_2),F_\beta(x_2)\}.
\end{align*}
Here $x_2\in J_\beta$: from $\beta-x_2^2>0$ we get $x_2<\sqrt\beta$, while $x_3<1$ and \eqref{eq:three-beta} give
\begin{align*}
  \beta<x_2^2+2x_1(1-x_2)<x_2^2+2x_2(1-x_2)=2x_2-x_2^2,
\end{align*}
equivalently $x_2>1-\sqrt{1-\beta}$.  Thus $C_\beta(x_2)\ge M(\beta)$ by definition, and $F_\beta(x_2)\ge M(\beta)$ by Lemma \ref{lem:one-two-gaps}.  Hence $\Phi_{x_2,\beta}(x_1)\ge M(\beta)$.
\end{proof}

\begin{lemma}\label{lem:no-interior-xrless1}
Let $\mathbf{x}=(x_0,\ldots,x_r)$ be a strict pattern satisfying $0=x_0<x_1<\cdots<x_r<1$.  If $\mathbf{x}$ is an interior local minimum of $F_r$ among $r$-class patterns with fixed edge density $\beta_r$, then
$
  F_r(\mathbf{x})\ge M(\beta_r(\mathbf{x})).
$
\end{lemma}

\begin{proof}
For $r\le3$, Lemmas \ref{lem:one-two-gaps} and \ref{lem:three-gaps} show that no such counterexample exists.  Assume $r\ge4$ and suppose, for a contradiction, that such a counterexample is an interior local minimum at fixed edge density.

Keep $x_2,\ldots,x_{r-1}$ fixed and vary only the two outside endpoints $x_1$ and $x_r$.  In the edge-density $\beta_r$, we have
\begin{align*}
  \beta_r=\sum_{i=1}^{r}(x_i-x_{i-1})x_{r+1-i}=x_1x_r-x_1x_{r-1}+x_1(x_r-x_{r-1})+x_2x_{r-1}+\sum_{i=3}^{r-1}(x_i-x_{i-1})x_{r+1-i}.
\end{align*}
Thus, preserving the edge density is equivalent to keeping
$
  K\coloneqq x_1(x_r-x_{r-1})
$
constant.  Thus, as $x_r$ varies,
\begin{align}\label{eq:x1-derivs}
  x_1(x_r)=\frac{K}{x_r-x_{r-1}},
  \qquad
  \frac{dx_1}{dx_r}=-\frac{x_1}{x_r-x_{r-1}},
  \qquad
  \frac{d^2x_1}{dx_r^2}=\frac{2x_1}{(x_r-x_{r-1})^2}.
\end{align}
For the normalized $S_{2,1}$-density,
\begin{align*}
  F_r(\mathbf{x})=\sum_{i=1}^r (x_i-x_{i-1})\psi(x_{r+1-i}).
\end{align*}
Under the same move, all terms are constant except
\begin{align*}
  G(x_r)\coloneqq(x_r-x_{r-1})\psi(x_1(x_r))+x_1(x_r)\psi(x_r)-x_1(x_r)\psi(x_{r-1})
\end{align*}
where $x_1=K/(x_r-x_{r-1})$.  Differentiating twice and substituting
$\psi''(t)=2-6t$ gives
\begin{align}\label{eq:G-second}
  G''(x_r)=\frac{2x_1(3x_1^2-x_1+(x_r-x_{r-1})^2)}{x_{r-1}-x_r}.
\end{align}
Since $x_{r-1}<x_r$ and local minimality forces $G''(x_r)\ge0$, we have
\begin{align}\label{eq:outside-ineq}
  3x_1^2-x_1+(x_r-x_{r-1})^2\le0.
\end{align}
Hence $3x_1^2-x_1<0$ and $x_1<1/3$.  Also, \eqref{eq:outside-ineq} gives
\begin{align*}
  (x_r-x_{r-1})^2\le x_1-3x_1^2\le \max_{0\le t\le1/3}(t-3t^2)=\frac1{12}.
\end{align*}
Therefore
\begin{align}\label{eq:gap-small}
  x_r-x_{r-1}\le\frac1{2\sqrt3}.
\end{align}
By Lemma \ref{lem:first-variation}, the quantities $A(x_{r-j},x_{r+1-j})+\psi'(x_j)$ are all equal.  Taking $j=1$ and $j=r$ gives
\begin{align}\label{eq:balance-1r}
  A(x_{r-1},x_r)+\psi'(x_1)=A(0,x_1)+\psi'(x_r).
\end{align}
Equation \eqref{eq:balance-1r} becomes
\begin{align}\label{eq:balance-simplified}
  (x_r-x_{r-1})(1-3x_r+x_r-x_{r-1})=x_1(1-2x_1).
\end{align}
We have $x_1(1-2x_1)\ge0$ since $0<x_1<1/3$, and $x_r-x_{r-1}>0$.  Hence
$
  1-3x_r+x_r-x_{r-1}>0.
$
Combining this with \eqref{eq:gap-small} gives
\begin{align}\label{eq:xr-less-half}
  x_r<\frac{1+x_r-x_{r-1}}{3}
  \le \frac{1+1/(2\sqrt3)}{3}<\frac12.
\end{align}
All positive degree levels are at most $x_r$, and their total weight is $\sum_i\Delta_i=x_r$.  Hence $1-x_{r+1-i}\ge1-x_r$, and
\begin{align*}
  F_r=\sum_i \Delta_i x_{r+1-i}^2(1-x_{r+1-i})
  \ge (1-x_r)\sum_i \Delta_i x_{r+1-i}^2.
\end{align*}
By Cauchy's inequality,
\begin{align*}
  \beta_r^2=\left(\sum_i\Delta_i x_{r+1-i}\right)^2
  \le\left(\sum_i\Delta_i\right)\left(\sum_i\Delta_i x_{r+1-i}^2\right)
  =x_r\sum_i\Delta_i x_{r+1-i}^2.
\end{align*}
Therefore
\begin{align}\label{eq:F-ge-beta-square}
  F_r\ge \frac{1-x_r}{x_r}\beta_r^2>\beta_r^2,
\end{align}
where the last inequality uses \eqref{eq:xr-less-half}.  On the other hand, Facts \ref{fact:qs-endpoint} and \ref{fact:cc-square-bound} give
\begin{align*}
  M(\beta_r)\le\cc(\beta_r)\le\frac{\beta_r^2}{4}<\beta_r^2<F_r,
\end{align*}
where $\beta_r>0$ since the pattern is strict.  This contradicts the counterexample assumption $F_r<M(\beta_r)$.
\end{proof}

\begin{lemma}\label{lem:no-interior-xreq1}
Let $\mathbf{x}=(x_0,\ldots,x_r)$ be a strict pattern satisfying $0=x_0<x_1<\cdots<x_{r-1}<x_r=1$.  If $\mathbf{x}$ is a relative interior local minimum of $F_r$ on the face $x_r=1$, at fixed edge density $\beta_r$, then
$
  F_r(\mathbf{x})\ge M(\beta_r(\mathbf{x})).
$
\end{lemma}

\begin{proof}
We first spell out the cases $r\le3$.  Recall that 
\begin{align*}
  \beta_r({\bf x})=\sum_{i=1}^{r}(x_i-x_{i-1})x_{r+1-i}\quad\text{ and }\quad F_r({\bf x})=\sum_{i=1}^{r}(x_i-x_{i-1})\psi(x_{r+1-i})
\end{align*}If $r=1$, then $x_1=1$ and $F_1({\bf x})=x_1\psi(x_1)=0=M(1)$.  If $r=2$, write $x_2=1$.  Then
\begin{align*}
  \beta\coloneqq\beta_2({\bf x})=2x_1-x_1^2,
  \qquad
  F_2({\bf x})=(1-x_1)\psi(x_1)=x_1^2(1-x_1)^2.
\end{align*}
Thus $x_1=1-\sqrt{1-\beta}$, the left endpoint in Fact \ref{fact:qs-endpoint}.  Hence
\begin{align*}
  F_2({\bf x})=(1-\beta)(1-\sqrt{1-\beta})^2=C_\beta(x_1)=\cc(\beta)\ge M(\beta).
\end{align*}
If $r=3$, put $\beta\coloneqq\beta_3({\bf x})$, $h\coloneqq x_1$ and
$z\coloneqq x_2$.  The three gaps have limiting weights $h$, $z-h$ and
$1-z$, exactly the classes $A,B,C$ in the complement-split construction.  The
edge density is
\begin{align*}
  \beta=z^2+2h(1-z).
\end{align*}
Hence $h=(\beta-z^2)/(2(1-z))=h_\beta(z)$.  Since $0\le h\le z<1$,
\eqref{eq:Jbeta-hbeta} gives $z\in J_\beta$.  Using \eqref{eq:Cbeta}, the
normalized $S_{2,1}$-density is
\begin{align*}
  F_3({\bf x})=(z-h)z^2(1-z)+(1-z)h^2(1-h)=C_\beta(z)\ge M(\beta).
\end{align*}
It remains to exclude relative interior local-minimum counterexamples with $r\ge4$.

\medskip
\noindent\textbf{Four positive gaps.}
Let $0<x_1<x_2<x_3<x_4=1$.  Since $x_4=1$, the edge-density formula gives
\begin{align}
 \beta\coloneqq \beta_4({\bf x})
  &=x_1+(x_2-x_1)x_3+(x_3-x_2)x_2+(1-x_3)x_1=2x_1-2x_1x_3+2x_2x_3-x_2^2.
\end{align}
Writing $x_2=x_1+\Delta_2$ and $x_3=x_1+\Delta_2+\Delta_3$, this becomes
\begin{align}\label{eq:four-beta}
  \beta_4({\bf x})=2x_1-x_1^2+\Delta_2^2+2\Delta_2\Delta_3.
\end{align}
Also
\begin{align}\label{eq:four-F}
  F_4({\bf x})=\Delta_3\psi(x_2)+(1-x_3)\psi(x_1)+\Delta_2\psi(x_3).
\end{align}
First, we claim that $x_2>1/3$.  By Lemma \ref{lem:first-variation}, with $x_4=1$ fixed, the free-coordinate balance equations give
\begin{align}\label{eq:four-balance-first}
  A(x_3,1)+\psi'(x_1)=A(x_2,x_3)+\psi'(x_2).
\end{align}
Substituting $A(a,b)=a+b-a^2-ab-b^2$ and $\psi'(t)=2t-3t^2$ into
\eqref{eq:four-balance-first} gives
\begin{align}\label{eq:four-x2-third}
  x_3(1-x_2)=2x_1-3x_1^2+4x_2^2-3x_2.
\end{align}
Suppose instead that $x_2\le1/3$.  Since $0<x_1<x_2\le1/3$ and
$\psi'(t)=2t-3t^2$ is increasing on $[0,1/3]$, we have
$
  2x_1-3x_1^2<2x_2-3x_2^2.
$
Hence the right-hand side of \eqref{eq:four-x2-third} satisfies
\begin{align*}
  2x_1-3x_1^2+4x_2^2-3x_2
  <2x_2-3x_2^2+4x_2^2-3x_2
  =x_2^2-x_2<0.
\end{align*}
But the left-hand side of \eqref{eq:four-x2-third} is $x_3(1-x_2)>0$, a contradiction.  Thus $x_2>1/3$.

Next we prove that $\Delta_2>\Delta_3$.  Again by Lemma \ref{lem:first-variation},
\begin{align}\label{eq:four-balance-second}
  A(x_2,x_3)+\psi'(x_2)=A(x_1,x_2)+\psi'(x_3).
\end{align}
Writing $x_1=x_2-\Delta_2$ and $x_3=x_2+\Delta_3$, then expanding and simplifying
\eqref{eq:four-balance-second}, we obtain
\begin{align}\label{eq:Efourdiff}
  (\Delta_2-\Delta_3)(3x_2-1)=\Delta_2^2+2\Delta_3^2.
\end{align}
Since $x_2>1/3$ and $\Delta_2^2+2\Delta_3^2>0$, \eqref{eq:Efourdiff} forces
$\Delta_2>\Delta_3$.  Hence $0<\Delta_3<\Delta_2$.

Now fix $\Delta_1$ and the value of
$\kappa\coloneqq\Delta_2^2+2\Delta_2\Delta_3.$
By \eqref{eq:four-beta}, this keeps $\beta$ fixed.  Along
\begin{align}\label{eq:Delta3-curve}
  \Delta_3(\Delta_2)=\frac{\kappa-\Delta_2^2}{2\Delta_2},
\end{align}
we preserve $\beta$.  Let primes denote differentiation with respect to $\Delta_2$ along this curve.  From $\Delta_2^2+2\Delta_2\Delta_3=\kappa$,
\begin{align*}
  \Delta_3'=-\frac{\Delta_2+\Delta_3}{\Delta_2},
  \qquad
  \Delta_3''=\frac{\Delta_2+2\Delta_3}{\Delta_2^2}.
\end{align*}
Also $x_1$ is fixed, $x_2=x_1+\Delta_2$, and $x_3=x_1+\Delta_2+\Delta_3$, so
\begin{align*}
  x_2'=1,
  \qquad
  x_3'=1+\Delta_3'=-\frac{\Delta_3}{\Delta_2},
  \qquad
  x_3''=\Delta_3''=\frac{\Delta_2+2\Delta_3}{\Delta_2^2}.
\end{align*}
View \eqref{eq:four-F} as the one-variable function
\begin{align*}
  F(\Delta_2)=\Delta_3\psi(x_2)+(1-x_3)\psi(x_1)+\Delta_2\psi(x_3),
\end{align*}
where $\Delta_3$ is given by \eqref{eq:Delta3-curve}.  Applying the chain rule gives
\begin{align}\label{eq:four-gap-Fpp-before}
F''={}&\frac{\Delta_2+2\Delta_3}{\Delta_2^2}(\psi(x_2)-\psi(x_1))
-\frac{2(\Delta_2+\Delta_3)}{\Delta_2}\psi'(x_2)+\Delta_3\psi''(x_2)+\psi'(x_3)+\frac{\Delta_3^2}{\Delta_2}\psi''(x_3).
\end{align}
Since $x_1=x_2-\Delta_2$ and $\psi'''(t)=-6$, the exact Taylor identity yields
\begin{align}\label{eq:taylor-psi-four}
  \psi(x_2)-\psi(x_1)=\Delta_2\psi'(x_2)-\frac{\Delta_2^2}{2}\psi''(x_2)-\Delta_2^3,
\end{align}
Substituting \eqref{eq:taylor-psi-four} and using $x_3=x_2+\Delta_3$, $\psi'(t)=2t-3t^2$, and $\psi''(t)=2-6t$, we get
\begin{align*}
F''&=(3\Delta_1-1)\left(\Delta_2-2\Delta_3-\frac{2\Delta_3^2}{\Delta_2}\right)
+2\Delta_2^2-8\Delta_2\Delta_3-9\Delta_3^2-\frac{6\Delta_3^3}{\Delta_2}.
\end{align*}
Substituting $x_2=\Delta_1+\Delta_2$ into \eqref{eq:Efourdiff}, we obtain
\begin{align}\label{eq:delta1-rewrite}
  3\Delta_1-1=\frac{-2\Delta_2^2+3\Delta_2\Delta_3+2\Delta_3^2}{\Delta_2-\Delta_3}.
\end{align}
Putting \eqref{eq:delta1-rewrite} into \eqref{eq:four-gap-Fpp-before} and taking the common denominator $\Delta_2(\Delta_2-\Delta_3)$ yields
\begin{align}\label{eq:four-second-final}
  \frac{d^2F}{d\Delta_2^2}
  =-\frac{\Delta_3(3\Delta_2^3+\Delta_2^2\Delta_3+7\Delta_2\Delta_3^2-2\Delta_3^3)}{\Delta_2(\Delta_2-\Delta_3)}.
\end{align}
Since $0<\Delta_3<\Delta_2$, we have $\Delta_2(\Delta_2-\Delta_3)>0$.  The bracket in the numerator is also positive, since
\begin{align*}
  7\Delta_2\Delta_3^2-2\Delta_3^3
  =\Delta_3^2(7\Delta_2-2\Delta_3)=\Delta_3^2(5\Delta_2+2(\Delta_2-\Delta_3))>0,
\end{align*}
and the remaining terms $3\Delta_2^3+\Delta_2^2\Delta_3$ are positive.
Thus $d^2F/d\Delta_2^2<0$, contradicting local minimality.

\medskip
\noindent\textbf{Five positive gaps.}
Let $0<x_1<x_2<x_3<x_4<x_5=1$.  Again write $x_0=0$.  The edge density is
\begin{align}
  \beta=\beta_5({\bf x})
  &=\sum_{i=1}^{5}(x_i-x_{i-1})x_{r+1-i}=2x_1+x_3^2-2x_2x_3+2x_4(x_2-x_1),\label{eq:five-beta}
\end{align}
and the $S_{2,1}$-density is
\begin{align}\label{eq:five-F}
  F_5({\bf x})=(x_2-x_1)\psi(x_4)+(x_3-x_2)\psi(x_3)+(x_4-x_3)\psi(x_2)+(1-x_4)\psi(x_1).
\end{align}
By Lemma \ref{lem:first-variation}, applied with $r=5$ and $x_5=1$ fixed, taking $j=1$ and $j=3$ must be equal:
\begin{align}\label{eq:five-balance}
  A(x_4,1)+\psi'(x_1)=A(x_2,x_3)+\psi'(x_3).
\end{align}
We first claim that $x_3>1/3$. If $x_3\le1/3$.  In that case $x_1<x_2<x_3\le1/3$, so $\psi'(t)=2t-3t^2$ is increasing on this interval and thus $\psi'(x_1)<\psi'(x_3)$.  Also $A(x_2,x_3)>0$.  Indeed,
\begin{align}\label{eq:A23}
  A(x_2,x_3)=x_2+x_3-x_2^2-x_2x_3-x_3^2>(x_2+x_3)-(x_2+x_3)^2>0,
\end{align}
where the two inequalities follow from $x_2+x_3<1$ and thus $x_2^2+x_2x_3+x_3^2<(x_2+x_3)^2<x_2+x_3.$
Thus, 
\begin{align*}
  A(x_2,x_3)+\psi'(x_3)>\psi'(x_1).
\end{align*}
On the other hand,
\begin{align*}
  A(x_4,1)+\psi'(x_1)=-x_4^2+\psi'(x_1)<\psi'(x_1),
\end{align*}
contradicting \eqref{eq:five-balance}.  Hence $x_4>x_3>1/3$.

Fix $x_1,x_3$ and $\beta$, vary $x_2$, and determine
\begin{align}\label{eq:x4-function}
  x_4(x_2)=\frac{\beta-2x_1-x_3^2+2x_2x_3}{2(x_2-x_1)}.
\end{align}
Then
\begin{align}\label{eq:x4-derivatives}
  x_4'(x_2)=-\frac{x_4-x_3}{x_2-x_1},
  \qquad
  x_4''(x_2)=\frac{2(x_4-x_3)}{(x_2-x_1)^2}.
\end{align}
Differentiating \eqref{eq:five-F} twice along this curve gives
\begin{align*}
\frac{d^2F}{dx_2^2}={}&2x_4'\psi'(x_4)+(x_2-x_1)(x_4')^2\psi''(x_4)+(x_2-x_1)x_4''\psi'(x_4)\\
&+x_4''(\psi(x_2)-\psi(x_1))+2x_4'\psi'(x_2)+(x_4-x_3)\psi''(x_2).
\end{align*}
Substituting \eqref{eq:x4-derivatives}, the two $\psi'(x_4)$-terms cancel, and we get
\begin{align}\label{eq:five-gap-Fpp-before}
\frac{d^2F}{dx_2^2}={}&(x_4-x_3)\psi''(x_2)-\frac{2(x_4-x_3)}{x_2-x_1}\psi'(x_2)+\frac{(x_4-x_3)^2}{x_2-x_1}\psi''(x_4)
+\frac{2(x_4-x_3)}{(x_2-x_1)^2}(\psi(x_2)-\psi(x_1)).
\end{align}
Since $\psi'''(t)=-6$, the Taylor expansion of $\psi$ at $x_2$ gives the exact identity
\begin{align}\label{eq:taylor-psi-five}
  \psi(x_2)-\psi(x_1)=(x_2-x_1)\psi'(x_2)-\frac{(x_2-x_1)^2}{2}\psi''(x_2)-(x_2-x_1)^3.
\end{align}
Using \eqref{eq:taylor-psi-five} and $\psi''(x_4)=2-6x_4=-2(3x_4-1)$,  \eqref{eq:five-gap-Fpp-before} reduces to
\begin{align}\label{eq:five-second-final}
  \frac{d^2F}{dx_2^2}
  =-\frac{2(x_4-x_3)}{x_2-x_1}\left((x_2-x_1)^2+(x_4-x_3)(3x_4-1)\right)<0,
\end{align}
since $x_4>x_3>1/3$ and $x_2>x_1$.  This contradicts local minimality.

\medskip
\noindent\textbf{Six or more positive gaps.}
Assume $r\ge6$.  Again write $\Delta_i=x_i-x_{i-1}$.  Since $x_r=1$, the edge density is
\begin{align}\label{eq:six-beta}
  \beta_r({\bf x})
  &=\sum_{i=1}^r\Delta_i x_{r+1-i}=\Delta_1x_r+\Delta_2x_{r-1}+\cdots+\Delta_{r-1}x_2+\Delta_rx_1,
\end{align}
and, since $\psi(1)=0$, the $S_{2,1}$-density is
\begin{align}\label{eq:six-F}
  F_r({\bf x})
  &=\sum_{i=1}^r\Delta_i\psi(x_{r+1-i})=\Delta_2\psi(x_{r-1})+\Delta_3\psi(x_{r-2})+\cdots+\Delta_r\psi(x_1).
\end{align}
By Lemma \ref{lem:first-variation}, applied with $x_r=1$ fixed, the quantities corresponding to $j=1$ and $j=r-2$ must be equal:
\begin{align}\label{eq:six-balance}
  A(x_{r-1},1)+\psi'(x_1)=A(x_2,x_3)+\psi'(x_{r-2}).
\end{align}
If $x_{r-2}\le1/3$, then $x_2<x_3<x_{r-2}\le1/3$.  Since $\psi'$ is increasing on $[0,1/3]$, we have $\psi'(x_{r-2})>\psi'(x_1)$.  Also $A(x_2,x_3)>0$ by the same argument as \eqref{eq:A23}.  Thus the right-hand side of \eqref{eq:six-balance} is greater than $\psi'(x_1)$, whereas
\begin{align*}
  A(x_{r-1},1)+\psi'(x_1)=-x_{r-1}^2+\psi'(x_1)<\psi'(x_1),
\end{align*}
a contradiction.  Therefore
\begin{align}\label{eq:high-levels}
  x_{r-1}>x_{r-2}>1/3.
\end{align}
Now perform an edge-preserving move of the threshold point classes.  Let $\mathbf{y}(t)=(y_1(t),\ldots,y_r(t))$ be the endpoint list after the move.  Increase $x_{r-2}$ by $t\Delta_2$, decrease $x_{r-1}$ by $t\Delta_3$, and keep all other endpoints fixed:
\begin{align*}
  y_i(t)=x_i\quad(i\ne r-2,r-1),
  \qquad
  y_{r-2}(t)=x_{r-2}+t\Delta_2,
  \qquad
  y_{r-1}(t)=x_{r-1}-t\Delta_3.
\end{align*}
For sufficiently small $t$, all point classes still have positive weight.  The choice of speeds comes from the anti-diagonal pairing: the weight $\Delta_3$ is paired with the degree level $x_{r-2}$, and the weight $\Delta_2$ is paired with the degree level $x_{r-1}$.  By~\eqref{eq:six-beta}, when all other endpoints are fixed, the part of the edge-density sum containing $x_{r-2}$ or $x_{r-1}$ is
\begin{align}\label{eq:rge6-edge-part}
  &\Delta_3x_{r-2}+\Delta_2x_{r-1}+(x_{r-2}-x_{r-3})x_3+(x_{r-1}-x_{r-2})x_2+(1-x_{r-1})x_1\notag\\
  &=2\Delta_3x_{r-2}+2\Delta_2x_{r-1}+x_1-x_3x_{r-3}.
\end{align}
This expression is linear in the two moved endpoints.  The coefficient of $x_{r-2}$ is $2\Delta_3$, and the coefficient of $x_{r-1}$ is $2\Delta_2$.
Hence the change in edge density is exactly
\begin{align*}
  2\Delta_3(t\Delta_2)+2\Delta_2(-t\Delta_3)=0,
\end{align*}
so $\beta_r(\mathbf{y}(t))=\beta_r(\mathbf{x})$ for all sufficiently small $t$.

Along $\mathbf{y}(t)$, the part of $F_r(\mathbf{y}(t))$ that depends on $t$ is
\begin{align}\label{eq:rge6-Fr-dependent}
&\Delta_3\psi(x_{r-2}+t\Delta_2)+\Delta_2\psi(x_{r-1}-t\Delta_3)+(x_{r-2}+t\Delta_2-x_{r-3})\psi(x_3)\notag\\
&+(x_{r-1}-t\Delta_3-x_{r-2}-t\Delta_2)\psi(x_2)+(1-x_{r-1}+t\Delta_3)\psi(x_1),
\end{align}
up to terms independent of $t$.  The last three terms in  \eqref{eq:rge6-Fr-dependent} are linear in $t$, so their second derivative is zero.  The first two terms give
\begin{align*}
  \left.\frac{d^2}{dt^2}\Delta_3\psi(x_{r-2}+t\Delta_2)\right|_{t=0}
  =\Delta_3\psi''(x_{r-2})\Delta_2^2,
\quad\text{ and }\quad
  \left.\frac{d^2}{dt^2}\Delta_2\psi(x_{r-1}-t\Delta_3)\right|_{t=0}
  =\Delta_2\psi''(x_{r-1})\Delta_3^2.
\end{align*}
Therefore
\begin{align}\label{eq:six-second}
  \left.\frac{d^2}{dt^2}F_r(\mathbf{y}(t))\right|_{t=0}
  =\Delta_3\psi''(x_{r-2})\Delta_2^2+\Delta_2\psi''(x_{r-1})\Delta_3^2<0,
\end{align}
since $x_{r-1}>x_{r-2}>1/3$ and $\psi''(s)=2-6s<0$ for $s>1/3$.  This exact edge-preserving move lowers $F_r$ to second order, which is impossible at a relative interior local minimum.  Therefore no strict relative-interior local-minimum counterexample with $x_r=1$ exists.
\end{proof}

We now assemble the threshold lower theorem from the finite-dimensional 
estimates.

\begin{proof}[Proof of Lemma \ref{prop:threshold-lower}]
First prove the finite-dimensional estimate
\begin{align}\label{eq:finite-dimensional-estimate}
  F_r(\mathbf{x})\ge M(\beta_r(\mathbf{x}))
\end{align}
for every finite $r\ge0$ and every weak closed staircase pattern
\begin{align*}
  0=x_0\le x_1\le\cdots\le x_r\le1.
\end{align*}
If not, choose a counterexample with minimal $r$, and then minimize
\begin{align*}
  D_r(\mathbf{x})\coloneqq F_r(\mathbf{x})-M(\beta_r(\mathbf{x}))
\end{align*}
on the compact set $0=x_0\le\cdots\le x_r\le1$.  This minimum is negative.  If $\Delta_j=0$, then Lemma \ref{lem:zero-gap} gives a shorter counterexample with the same $\beta$ and $F$, contradiction.  Thus all gaps are positive.  Moreover, $\mathbf{x}$ is a local minimum of $F_r$ under the fixed edge-density constraint $\beta_r=\beta_r(\mathbf{x})$: for every nearby $\mathbf{y}$ with $\beta_r(\mathbf{y})=\beta_r(\mathbf{x})$,
\begin{align*}
  F_r(\mathbf{y})-M(\beta_r(\mathbf{x}))=D_r(\mathbf{y})\ge D_r(\mathbf{x})=F_r(\mathbf{x})-M(\beta_r(\mathbf{x})).
\end{align*}
If $x_r<1$, this contradicts Lemma \ref{lem:no-interior-xrless1}; if $x_r=1$, it contradicts Lemma \ref{lem:no-interior-xreq1}.  Hence \eqref{eq:finite-dimensional-estimate} holds.

For $G_n$, take its full normalized endpoint list
\begin{align*}
  \mathbf{x}^{(n)}=(x_0^{(n)},\ldots,x_{r_n}^{(n)}).
\end{align*}
Each positive twin class is nonempty, so $r_n\le n$.  Thus
$\mathbf{x}^{(n)}$ is a finite weak closed staircase pattern, and
\eqref{eq:finite-dimensional-estimate} applies with $r=r_n$.  Moreover,
all finite gaps
\begin{align*}
  x_i^{(n)}-x_{i-1}^{(n)}=\frac{|V_i|}{n}
\end{align*}
are positive.  Hence no zero-gap endpoint deletion is involved at this stage:
we apply the finite-dimensional estimate to the full list of positive point
classes.  It gives
\begin{align}\label{eq:finite-pattern-applied}
  F_{r_n}(\mathbf{x}^{(n)})
  \ge M(\beta_{r_n}(\mathbf{x}^{(n)})).
\end{align}
On the other hand, \eqref{eq:finite-F} identifies
$F_{r_n}(\mathbf{x}^{(n)})$, up to an $O(1/n)$ error, with the normalized
$S_{2,1}$-count of $G_n$.  Combining \eqref{eq:finite-pattern-applied} with \eqref{eq:finite-F} and \eqref{eq:threshold-Fr-asymptotic} gives
\begin{align*}
  \frac{N(S_{2,1},G_n)}{n^4}
  \ge M(\beta_{r_n}(\mathbf{x}^{(n)}))-O(1/n),
\end{align*}
where, by \eqref{eq:finite-beta},
\begin{align*}
  \beta_{r_n}(\mathbf{x}^{(n)})=\frac{2e(G_n)}{n^2}+O(1/n)\to\beta.
\end{align*}
Therefore, by continuity of $M$,
\begin{align*}
  \liminf_{n\to\infty}\frac{N(S_{2,1},G_n)}{n^4}
  \ge M(\beta)=\min_{z\in J_\beta}C_\beta(z).
\end{align*}
This proves Lemma \ref{prop:threshold-lower}.
\end{proof}

\section*{Declaration on the use of AI}
The authors used generative AI tools to assist in discussing proof strategies, checking proofs, and improving exposition.

\bibliographystyle{alpha}
\bibliography{s21_profiles_referee_submission_final}

\newcommand{\etalchar}[1]{$^{#1}$}
\begin{thebibliography}{BLM{\etalchar{+}}26}

\bibitem[AK78]{AhlswedeKatona}
Rudolf Ahlswede and Gyula O.~H. Katona.
\newblock Graphs with maximal number of adjacent pairs of edges.
\newblock {\em Acta Mathematica Academiae Scientiarum Hungaricae}, 32:97--120,
  1978.

\bibitem[BEHJ95]{BollobasEgawaHarrisJin1995}
B.~Bollob{\'a}s, Y.~Egawa, A.~Harris, and G.~Jin.
\newblock The maximal number of induced {$r$}-partite subgraphs.
\newblock {\em Graphs Combin.}, 11:1--19, 1995.

\bibitem[BGH{\etalchar{+}}25]{BGHKS}
A.~Basit, B.~Granet, D.~Horsley, A.~K{\"u}ndgen, and K.~Staden.
\newblock The semi-inducibility problem, 2025.
\newblock arXiv:2501.09842.

\bibitem[BGL{\etalchar{+}}26]{BodnarGaoLeonLiuPikhurkoSun2026}
L.~Bodn{\'a}r, J.~Gao, J.~Le{\'o}n, X.~Liu, O.~Pikhurko, and S.~Sun.
\newblock The inducibility of 6-vertex graphs, 2026.
\newblock arXiv:2606.00290.

\bibitem[BLM{\etalchar{+}}26]{BLMPV}
J.~Balogh, B.~Lidick{\'y}, D.~Mubayi, F.~Pfender, and J.~Volec.
\newblock Semi-inducibility of some small graphs, 2026.
\newblock arXiv:2601.03433.

\bibitem[BP25a]{BodnarPikhurkoSemi2025}
L.~Bodn{\'a}r and O.~Pikhurko.
\newblock Semi-inducibility of 4-vertex graphs, 2025.
\newblock arXiv:2510.24336.

\bibitem[BP25b]{BodnarPikhurko}
L.~Bodn{\'a}r and O.~Pikhurko.
\newblock Some exact inducibility-type results for graphs via flag algebras,
  2025.
\newblock arXiv:2507.01596.

\bibitem[BS94]{BrownSidorenko1994}
J.~I. Brown and A.~Sidorenko.
\newblock The inducibility of complete bipartite graphs.
\newblock {\em J. Graph Theory}, 18(6):629--645, 1994.

\bibitem[CN25]{CN25}
H.~Chen and J.~A. Noel.
\newblock On alternating 6-cycles in edge-coloured graphs.
\newblock 2025.
\newblock arXiv:2505.09809.

\bibitem[HHN14]{HatamiHirstNorine2014}
H.~Hatami, J.~Hirst, and S.~Norine.
\newblock The inducibility of blow-up graphs.
\newblock {\em J. Combin. Theory Ser. B}, 109:196--212, 2014.

\bibitem[Kat68]{Katona1968}
Gyula O.~H. Katona.
\newblock A theorem of finite sets.
\newblock {\em Theory of Graphs}, pages 187--207, 1968.

\bibitem[Kru63]{Kruskal1963}
Joseph~B. Kruskal.
\newblock The number of simplices in a complex.
\newblock {\em Mathematical Optimization Techniques}, pages 251--278, 1963.

\bibitem[LM21]{LiuMubayiFeasible}
X.~Liu and D.~Mubayi.
\newblock The feasible region of hypergraphs.
\newblock {\em J. Combin. Theory Ser. B}, 148:23--59, 2021.

\bibitem[LMR23]{LiuMubayiReiher}
Xizhi Liu, Dhruv Mubayi, and Christian Reiher.
\newblock The feasible region of induced graphs.
\newblock {\em Journal of Combinatorial Theory, Series B}, 158:105--135, 2023.

\bibitem[LS06]{LovaszSzegedy}
L.~Lov{\'a}sz and B.~Szegedy.
\newblock Limits of dense graph sequences.
\newblock {\em J. Combin. Theory Ser. B}, 96:933--957, 2006.

\bibitem[PG75]{PippengerGolumbic}
N.~Pippenger and M.~Golumbic.
\newblock The inducibility of graphs.
\newblock {\em J. Combin. Theory Ser. B}, 19:189--203, 1975.

\bibitem[Raz07]{Razborov}
A.~A. Razborov.
\newblock Flag algebras.
\newblock {\em J. Symbolic Logic}, 72:1239--1282, 2007.

\bibitem[Raz08]{RazborovTriangle}
Alexander~A. Razborov.
\newblock On the minimal density of triangles in graphs.
\newblock {\em Combinatorics, Probability and Computing}, 17(4):603--618, 2008.

\bibitem[Rud76]{Rudin}
W.~Rudin.
\newblock {\em Principles of Mathematical Analysis}.
\newblock McGraw-Hill, 3 edition, 1976.

\bibitem[RW18]{ReiherWagner}
C.~Reiher and S.~Wagner.
\newblock Maximum star densities.
\newblock {\em Studia Sci. Math. Hungar.}, 55:238--259, 2018.

\end{thebibliography}

\appendix
\section{Proof of Proposition \ref{lem:split-quasiclique-comparison}}
\begin{proof}[Proof of Proposition \ref{lem:split-quasiclique-comparison}]

 Let
 \[
 Q(t):=6t^5-384t^4+338t^3-89t^2-4t+2
 \]
 be the polynomial in \eqref{eq:switching-equation}. 

Recall that $c(\beta)=\beta^{3/2}(1-\sqrt{\beta})$. It remains to compare \(s(\tau^2)\) with \(c(\tau^2)\). Put
\[
\Phi(x,y):=xy^2(1-y)+y(x+y)^2(1-x-y).
\]
Let \(0<\tau\le 1/2\). By definition,
\[
s(\tau^2)
=
\max\{\Phi(x,y): x,y\ge0,\ x+y\le1,\ 2xy+y^2=\tau^2\}.
\]
The point \((x,y)=(0,\tau)\) is feasible and gives
\[
\Phi(0,\tau)=\tau^3(1-\tau)=c(\tau^2).
\]

Let \((x,y)\) be feasible. Since \(\tau>0\), we have \(y>0\). Write
$y=\tau\lambda $. The constraint \(2xy+y^2=\tau^2\) gives
\begin{align}\label{eq:x}
x=\frac{\tau(1-\lambda^2)}{2\lambda}.
\end{align}
The feasibility conditions \(x\ge0\) and \(x+y\le1\) are equivalent to
\begin{equation}\label{eq:app-transition-feasible}
0<\lambda\le1,
\qquad
\tau\le \frac{2\lambda}{1+\lambda^2}.
\end{equation}
Conversely, every \(\lambda\) satisfying
\eqref{eq:app-transition-feasible} gives a feasible pair by the above
formulae.

Substituting~\eqref{eq:x} and $y=\tau\lambda$
into \(\Phi(x,y)\), a direct expansion gives
\begin{equation}\label{eq:app-transition-compare}
\Phi\left(\frac{\tau(1-\lambda^2)}{2\lambda},\tau\lambda\right)
-
c(\tau^2)
=
\frac{\tau^3(1-\lambda)^2}{8\lambda^2}T_\tau(\lambda),
\end{equation}
where
\[
T_\tau(\lambda)=B(\lambda)-\tau D(\lambda),
\qquad
B(\lambda)=-2\lambda((\lambda+1)^2-2),
\qquad
D(\lambda)=(1+\lambda)^2(1-3\lambda^2).
\]
For \(0<\lambda<1\), the prefactor in
\eqref{eq:app-transition-compare} is positive, so the sign of the difference
is the sign of \(T_\tau(\lambda)\). At \(\lambda=1\), the difference is zero.

On the interval \(0<\lambda<\sqrt{2}-1\), both \(B(\lambda)\) and
\(D(\lambda)\) are positive. Define
\begin{equation}\label{eq:app-transition-gdef}
g(\lambda):=\frac{B(\lambda)}{D(\lambda)}
=
\frac{-2\lambda((\lambda+1)^2-2)}
{(1+\lambda)^2(1-3\lambda^2)}.
\end{equation}
Then
\[
T_\tau(\lambda)=D(\lambda)(g(\lambda)-\tau),
\]
and therefore
\begin{equation}\label{eq:app-transition-Tsign}
T_\tau(\lambda)>0
\quad\Longleftrightarrow\quad
\tau<g(\lambda)
\qquad
(0<\lambda<\sqrt{2}-1).
\end{equation}

A direct differentiation gives
\begin{equation}\label{eq:app-transition-gderiv}
g'(\lambda)
=
-\frac{2H(\lambda)}
{(1+\lambda)^3(3\lambda^2-1)^2},
\end{equation}
where
\[
H(\lambda)=3\lambda^5+9\lambda^4-8\lambda^3+5\lambda-1.
\]
Moreover,
\[
H'(\lambda)=15\lambda^4+36\lambda^3-24\lambda^2+5.
\]
For \(0\le\lambda\le\sqrt{2}-1\), since \(\sqrt{2}-1<1/2\), we have
\[
H'(\lambda)\ge 5-24\lambda^2
\ge 5-24(\sqrt{2}-1)^2>0.
\]
Also
$H(0)=-1$ and
$H(\sqrt{2}-1)=80-56\sqrt{2}>0.$ 
Hence \(H\) has a unique zero \(\lambda_0\) in
\((0,\sqrt{2}-1)\). Since the denominator in
\eqref{eq:app-transition-gderiv} is positive, it follows that \(g\) is
increasing on \((0,\lambda_0)\) and decreasing on
\((\lambda_0,\sqrt{2}-1)\). Consequently,
\begin{equation}\label{eq:app-transition-gmax}
g(\lambda)\le g(\lambda_0)
\qquad
(0<\lambda<\sqrt{2}-1).
\end{equation}
Put
\[
m:=g(\lambda_0).
\]

We next identify \(m\). First, \(0<m<1/3\). Indeed, \(m>0\) because
\(B(\lambda)\) and \(D(\lambda)\) are positive on
\(0<\lambda<\sqrt{2}-1\). Also,
\[
D(\lambda)-3B(\lambda)
=
1-4\lambda+10\lambda^2-3\lambda^4
>
1-4\lambda+7\lambda^2>0
\]
on \(0<\lambda<\sqrt{2}-1\). Since \(D(\lambda)>0\) on this interval, this
gives
\[
g(\lambda)<\frac13
\qquad
(0<\lambda<\sqrt{2}-1),
\]
and hence \(m<1/3\).

Let
\[
F(\lambda,\gamma):=\gamma D(\lambda)-B(\lambda).
\]
Since \(F(\lambda,g(\lambda))\equiv0\), we have
$F(\lambda_0,m)=0$. 
Moreover, \(g'(\lambda_0)=0\). Differentiating
\(F(\lambda,g(\lambda))\equiv0\) with respect to \(\lambda\) and then
setting \(\lambda=\lambda_0\), we obtain
$F_\lambda(\lambda_0,m)=0$.

Thus \(\lambda_0\) is a multiple root of \(F(\lambda,m)\), viewed as a
polynomial in \(\lambda\). Since the leading coefficient of \(F(\lambda,m)\)
is \(-3m\ne0\), the quartic discriminant vanishes.

Expanding,
\[
F(\lambda,\gamma)
=
-3\gamma\lambda^4
+(2-6\gamma)\lambda^3
+(4-2\gamma)\lambda^2
+(2\gamma-2)\lambda
+\gamma .
\]
The discriminant with respect to \(\lambda\) is
\[
\operatorname{Disc}_\lambda F(\lambda,\gamma)
=
256Q(\gamma).
\]
Therefore $Q(m)=0$.
Since \(0<m<1/3\) and \(Q\) has a unique root in \((0,1/3)\), we get
$m=t_*$.
Combining this with \eqref{eq:app-transition-gmax}, we have
\begin{equation}\label{eq:app-transition-g-upper}
g(\lambda)\le t_*
\qquad
(0<\lambda<\sqrt{2}-1).
\end{equation}

We shall also use the following feasibility bound. A direct subtraction gives
\begin{equation}\label{eq:app-transition-feasibility-bound}
\frac{2\lambda}{1+\lambda^2}-g(\lambda)
=
\frac{4\lambda^2(2-\lambda-2\lambda^2-\lambda^3)}
{(1+\lambda)^2(1+\lambda^2)(1-3\lambda^2)}
>0
\end{equation}
for \(0<\lambda<\sqrt{2}-1\). Indeed, the denominator is positive because
\(\sqrt{2}-1<1/\sqrt{3}\). Also \(\sqrt{2}-1<1/2\), so
\[
2-\lambda-2\lambda^2-\lambda^3
>
2-\frac12-2\cdot\frac14-\frac18>0.
\]

Now assume first that \(0<\tau<t_*\). Choose \(\lambda=\lambda_0\). Since
\[
\tau<t_*=g(\lambda_0)
<
\frac{2\lambda_0}{1+\lambda_0^2}
\]
by \eqref{eq:app-transition-feasibility-bound}, the value \(\lambda_0\)
satisfies the feasibility condition \eqref{eq:app-transition-feasible}.
Moreover,
\[
T_\tau(\lambda_0)
=
D(\lambda_0)(t_*-\tau)>0,
\]
because \(D(\lambda_0)>0\). Since \(0<\lambda_0<1\), the prefactor in
\eqref{eq:app-transition-compare} is strictly positive. Therefore
\[
\Phi\left(\frac{\tau(1-\lambda_0^2)}{2\lambda_0},\tau\lambda_0\right)
>
c(\tau^2).
\]
Hence
\[
s(\tau^2)>c(\tau^2)
\qquad
(0<\tau<t_*).
\]

Finally assume that \(t_*\le\tau\le1/2\). We show that every feasible
\(\lambda\) gives value at most \(c(\tau^2)\). By
\eqref{eq:app-transition-compare}, it is enough to prove
\[
T_\tau(\lambda)\le0
\qquad
(0<\lambda\le1).
\]

First let \(0<\lambda<\sqrt{2}-1\). Then
\eqref{eq:app-transition-g-upper} gives
\[
g(\lambda)\le t_*\le\tau.
\]
Since \(D(\lambda)>0\), we obtain
\[
T_\tau(\lambda)
=
D(\lambda)(g(\lambda)-\tau)\le0.
\]

Next let \(\sqrt{2}-1\le\lambda<1/\sqrt{3}\). Then
\[
B(\lambda)\le0,
\qquad
D(\lambda)>0,
\]
and therefore
\[
T_\tau(\lambda)=B(\lambda)-\tau D(\lambda)<0.
\]

Finally let \(1/\sqrt{3}\le\lambda\le1\). Then \(D(\lambda)\le0\), so
\(T_\tau(\lambda)=B(\lambda)-\tau D(\lambda)\) is increasing as a function of
\(\tau\). Since \(\tau\le1/2\),
\[
T_\tau(\lambda)\le T_{1/2}(\lambda).
\]
A direct simplification gives
\[
T_{1/2}(\lambda)
=
\frac{\lambda-1}{2}
\left(\lambda(3\lambda^2-1)+5\lambda^2+1\right).
\]
For \(1/\sqrt{3}\le\lambda\le1\), the factor in parentheses is positive, and
\(\lambda-1\le0\). Hence
\[
T_{1/2}(\lambda)\le0,
\]
and so \(T_\tau(\lambda)\le0\).

Thus every feasible point has value at most \(c(\tau^2)\). The endpoint
\(\lambda=1\), equivalently \((x,y)=(0,\tau)\), is feasible and gives
equality. Therefore
\[
s(\tau^2)=c(\tau^2)
\qquad
(t_*\le\tau\le1/2).
\]

Since \(c(\tau^2)=\tau^3(1-\tau)=\beta^{3/2}(1-\sqrt{\beta})\) with
\(\beta=\tau^2\), the desired comparison follows.
\end{proof}

\end{document}